\RequirePackage{fix-cm}
\RequirePackage{amsmath}
\documentclass{svjour3}
\usepackage[normalem]{ulem}

\usepackage{natbib}
\usepackage{hyperref}
\hypersetup{
    colorlinks = true,
    urlcolor = blue,
    citecolor = blue,
    linkcolor = black
}
\usepackage{enumerate}
\usepackage{framed}
\usepackage[shortlabels,inline]{enumitem}

\usepackage{amsmath,amsfonts,amssymb} 
\usepackage{graphicx}
\usepackage{color}
\usepackage{subfigure}
\usepackage{algorithm}

\usepackage[hmargin=2.5cm,vmargin=3cm,bindingoffset=0.5cm]{geometry}

\usepackage {tikz}
\usepackage{pgfkeys}
\usetikzlibrary{shapes,arrows}
\usepackage{pgfplots}
\pgfplotsset{compat=1.13}

\pgfplotsset{plotOptions/.style={%
		width=\linewidth,
		xlabel={Iteration $k$},
		ylabel={$\frac{f(w_k)-f_\star}{\|w_0-w_\star\|^2}$},
		label style={font=\small},
		legend style={font=\small},
		tick label style={font=\small},
		solid,
		very thick
	}}
\pgfplotsset{plotOptions2/.style={%
		width=\linewidth,
		xlabel={Iteration $k$},
		ylabel={$\frac{f(w_k)-f_\star}{f(w_0)-f_\star}$},
		label style={font=\small},
		legend style={font=\small},
		tick label style={font=\small},
		solid,
		very thick
	}}	

\newcommand{\cond}{q}
\newcommand{\varS}{{\bar S}}

\newcommand{\bxx}{\mathbf{w}}
\newcommand{\bgg}{\mathbf{g}}
\newcommand{\bff}{\mathbf{f}}
\newcommand{\sz}{\lambda}
\newcommand{\fullMethodName}{{Information-Theoretic Exact Method}}
\newcommand{\shortMethodName}{{ITEM}}
\journalname{Journal Name}

\begin{document}

\title{An optimal gradient method for smooth strongly convex minimization\thanks{A. Taylor acknowledges support from the European Research Council (grant SEQUOIA 724063).This work was funded in part by the french government under management of Agence Nationale de la recherche as part of the ``Investissements d’avenir'' program, reference ANR-19-P3IA-0001 (PRAIRIE 3IA Institute).}}

\titlerunning{Optimal gradient methods}    

\author{Adrien Taylor \and Yoel Drori}

\authorrunning{A.~Taylor, Y.~Drori}

\institute{Adrien  Taylor  \at
              INRIA, D\'epartement d'informatique de l'ENS, \'Ecole normale sup\'erieure, CNRS, PSL Research University, Paris, France
              Email: adrien.taylor@inria.fr\\
              Yoel Drori \at Google Research Israel. Email: dyoel@google.com}

\date{Date of current version: \today}

\maketitle

\begin{abstract}We present an optimal gradient method for smooth strongly convex optimization. The method is optimal in the sense that its worst-case bound on the distance to an optimal point \emph{exactly} matches the lower bound on the oracle complexity for the class of problems, meaning that no black-box first-order method can have a better worst-case guarantee without further assumptions on the class of problems at hand.
In addition, we provide a constructive recipe for obtaining the algorithmic parameters of the method and illustrate that it can be used for deriving methods for other optimality criteria as well.
\end{abstract}
\section{Introduction}
Consider the unconstrained minimization problem
\begin{equation}\label{eq:OPT}\min_{x\in\mathbb{R}^d} f(x),\end{equation}
where $f$ is a smooth strongly convex function. For solving such problems, one can rely on black-box first-order methods, which iteratively acquire information about $f$ by evaluating its gradient at a sequence of iterates. In this context, the question of designing first-order methods with good worst-case guarantees occupies an important place. 

In this work, we provide a black-box first-order method, the \fullMethodName{ }(\shortMethodName),
designed for minimizing smooth strongly convex functions. This method attains the lower bound on the oracle complexity, sometimes referred to as information-theoretic complexity~\citep{nemirovsky1992information}, of smooth strongly convex minimization when optimality is measured by the distance of the method's output to an optimal solution. 

Given an $L$-smooth $\mu$-strongly convex function $f$ with $0<\mu<L$, the method can be concisely written as
\begin{equation}\label{eq:Method}
\begin{aligned}
{y_{k}}&=(1-\beta_k) z_k+\beta_k {\left(y_{k-1}-\frac1L \nabla f(y_{k-1})\right)} \\
z_{k+1}&={(1-\cond \delta_k)z_k+\cond \delta_k \left(y_{k}-\frac1\mu  \nabla f(y_{k})\right)},
\end{aligned}
\end{equation}
where $\cond=\mu/L$ denotes the inverse condition number (note that as $\mu\rightarrow 0$ the alternate formulation below should be preferred for obvious numerical reasons). Both sequences $\{\beta_k\}$ and $\{\delta_k\}$ are parametrized by a sequence $A_k$, incorporating the dependency on the current iteration number, as follows
\[ \beta_k=\frac{A_k}{(1-\cond)A_{k+1}}\text{, and }  \delta_k=\frac1{2}\frac{(1-\cond )^2 A_{k+1}-(1+\cond ) A_k}{1+\cond+\cond  A_k},\]
with $A_0=0$ and
\[ A_{k+1}= \frac{(1+\cond) A_k+2 \left(1+\sqrt{(1+A_k) (1+\cond  A_k)}\right)}{(1-\cond )^2}, \quad k\geq 0.\]
As shown in the following, this sequence allows to describe the worst-case performance of~\eqref{eq:Method} as 
\[ \|z_N-x_\star\|^2\leq \frac{1}{1+\cond A_N}\|z_0-x_\star\|^2,\]
which is the exact lower bound for smooth strongly convex minimization, as obtained in~\citep{drori2021exact}. Therefore, no black-box first-order method can further improve this guarantee, and \shortMethodName{} achieves the lower bound on the oracle (or information-theoretic) complexity for smooth strongly convex minimization. In addition, as $A_N\geq (1-\sqrt{q})^{-2N}$, this bound provides a guarantee that $z_k$ strictly improves over $z_0$ with a worst-case convergence rate $(1-\sqrt{q})^{2}$.

\shortMethodName{} is also closely related to other methods. In particular, when $\mu>0$ and as $k\rightarrow\infty$, the method's parameters $\beta_k$ and $\delta_k$ tends to those of the Triple Momentum Method (TMM) by~\citet{van2017fastest}, and in the case $\mu=0$ the parameters correspond to those of the Optimized Gradient Method (OGM) of~\citet{kim2015optimized}, which exactly achieves a lower complexity bound for minimizing function values as established in~\citet{drori2017exact}. Details on those relationships are provided in Section~\ref{s:limits}.

\subsection{Related works}
\paragraph{Lower bounds and accelerated methods.} The method presented in this work is closely related to the celebrated fast gradient methods (FGMs) by~\citet{nest83,nest-book-04}. Lyapunov and potential function-based analyses of FGMs were presented in many works, including in the original~\citep{nest83}. The analyses are usually tailored for the smooth convex minimization setting~\citep{nest83,beck2009fast}, for the smooth strongly convex one~\citep{wilson2016lyapunov,bansal2019potential}, and sometimes deal with both simultaneously~\citep{nest-book-04,gasnikov2018universal}. In the large-scale quadratic smooth (possibly strongly) convex minimization setting, optimal worst-case accuracies are achieved by Chebyshev and conjugate gradient methods~\citep{nemirovsky1992information,nemirovski1999optimization}. Fast first-order methods for large-scale convex minimization are surveyed in the recent monograph~\citep{dAspremontAccelerated}.


\paragraph{Performance estimation problems.} The idea of computing worst-case accuracy of a given method through semidefinite programming dates back to~\citet{drori2013performance}. It was {refined} using the concept of convex interpolation in~\citep{taylor2017smooth}, which allows guaranteeing that worst-case accuracies provided by the semidefinite programs are tight (i.e., the worst-case guarantee corresponds to a matching example in the problem class). The approach was taken further in different directions for analyzing and designing numerical methods in different contexts. A very related line of works, initiated by~\citep{lessard2016analysis}, presents such analyses from a control theoretic perspective, and corresponds to the problem of looking for Lyapunov functions. Those works hence rather target \emph{asymptotic} properties of time-invariant numerical methods, which allows using smaller sized SDPs. 

\paragraph{Optimized gradient methods.} The method presented {in this work} was first obtained as a solution to a convex optimization problem, through an approach closely related to that taken by~\citet{drori2013performance} and~\citet{kim2015optimized} for obtaining the Optimized Gradient Method (OGM). The OGM for smooth convex minimization ($\mu=0$), obtained by~\citet{kim2015optimized}, was obtained by explicitly choosing the step sizes of a method for minimizing an upper bound on the worst-case inaccuracy criterion.
The resulting method was later proved to achieve the lower bound in~\citep{drori2017exact}. When $\mu=0$, optimal methods for optimizing function value accuracy $f(x_N)-f_\star$ include the OGM~\citep{kim2015optimized,drori2017exact}, and the conjugate gradient method~\citep{drori2020efficient}. It is also worth mentioning that optimized methods can be developed for other criteria as well. In particular, optimized methods for gradient norms $\|\nabla f(x_N)\|^2$ are studied by~\citet{kim2018optimizing}, in the smooth convex setting. 

The \emph{Triple Momentum Method} (TMM)~\citep{van2017fastest} was designed as an optimized gradient method through Lyapunov arguments, using an idea similar to that of the OGM, but for time-independent methods (i.e., whose coefficients do not depend on the iteration counter), for when $\mu>0$. The method was originally obtained using the integral quadratic framework by~\citet{lessard2016analysis}; see~\citep{van2017fastest} and~\citep{lessard2020direct}. The problem of devising optimized methods for smooth strongly convex minimization (with $\mu>0$) is also addressed in~\citep{zhou2020boosting,gramlich2020convex}, which also recovers the TMM as a particular case in their analyses.

So far, it remained unclear how to conciliate both \emph{optimal} methods, as the OGM is clearly not optimal anymore when $\mu>0$ (its worst-case guarantees remain unchanged in the presence of strong convexity~\citep{kim2017convergence}), and as the TTM is not defined when $\mu=0$. 
 
\subsection{Organization}
A worst-case analysis of the \fullMethodName{ }is provided in Section~\ref{s:OGM_sc}. 
In Section~\ref{s:optimized_steps}, we describe a constructive approach that leads to the method and illustrate that it can be used for developing optimized methods for other performance criteria. 
We draw some conclusions in Section~\ref{s:ccl}.

\subsection{Preliminaries and notations}  
We use the standard notation $\langle\,\cdot\,;\,\cdot\,\rangle:\mathbb{R}^d\times\mathbb{R}^d\rightarrow \mathbb{R}$ to denote the Euclidean inner product, and the corresponding induced Euclidean norm $\|\cdot\|$. Furthermore, we denote by $x_\star$ some optimal solution to~\eqref{eq:OPT} ({which is} unique if $\mu>0$), and by $f_\star$ its optimal value. The class of $L$-smooth $\mu$-strongly convex functions is standard and can be defined as follows.
\begin{definition}Let $f:\mathbb{R}^d\rightarrow\mathbb{R}$ be a proper, closed, and convex function, and consider two constants $0\leq\mu<L<\infty$. We say that $f$ is $L$-smooth and $\mu$-strongly convex, denoted $f\in\mathcal{F}_{\mu,L}(\mathbb{R}^d)$, if
\begin{itemize}
    \item ($L$-smooth) for all $x,y\in\mathbb{R}^d$, it holds that $f(x)\leq f(y)+\langle \nabla f(y);x-y\rangle+\frac{L}{2}\|x-y\|^2$,
    \item ($\mu$-strongly convex)  for all $x,y\in\mathbb{R}^d$, it holds that $f(x)\geq f(y)+\langle \nabla f(y);x-y\rangle+\frac{\mu}{2}\|x-y\|^2$.
\end{itemize}
We simply denote $f\in\mathcal{F}_{\mu,L}$ when the dimension is either clear from the context or unspecified. In addition, we use $\cond:=\mu/L$ the (inverse) condition number of the class (hence $0\leq \cond< 1$), and do not explicitly treat the trivial cases $L=\mu$ for readability purposes.

\end{definition}
Smooth strongly convex functions satisfy many inequalities, see e.g.,~\citep[Theorem 2.1.5]{nest-book-04}. For the developments below, we need only one specific inequality characterizing functions in $\mathcal{F}_{\mu,L}$.
\begin{theorem}\label{thm:interp} Let $f\in\mathcal{F}_{\mu,L}(\mathbb{R}^d)$. For all $x,y\in\mathbb{R}^d$, it holds that
\begin{equation*}
\begin{aligned}
f(y)\geq f(x)+&\langle \nabla f(x);y-x\rangle+\frac{1}{2L}\|\nabla f(x)-\nabla f(y)\|^2+\frac{\mu}{2(1-\mu/L)}\|x-y-\tfrac{1}{L}(\nabla f(x)-\nabla f(y))\|^2.
\end{aligned}
\end{equation*}
\end{theorem}
This inequality turns out to be key in proving worst-case guarantees for first-order methods applied on smooth strongly convex problems, due to the following result~\citep[Theorem 4]{taylor2017smooth}.
\begin{theorem}[$\mathcal{F}_{\mu,L}$-interpolation]\label{thm:interp2} Let $I$ be an index set and $S=\{(x_i,g_i,f_i)\}_{i\in I}\subseteq \mathbb{R}^d\times\mathbb{R}^d\times \mathbb{R}$ be a set of triplets. There exists $f\in\mathcal{F}_{\mu,L}$ satisfying $f(x_i)=f_i$ and $g_i\in\partial f(x_i)$ for all $i\in I$ if and only if
\begin{equation*}
\begin{aligned}
f_i\geq f_j+&\langle g_j;x_i-x_j\rangle+\frac{1}{2L}\|g_i-g_j\|^2+\frac{\mu}{2(1-\mu/L)}\|x_i-x_j-\tfrac{1}{L}(g_i-g_j)\|^2
\end{aligned}
\end{equation*}
holds for all $i,j\in I$.
\end{theorem}

\section{An optimal gradient method}\label{s:OGM_sc}
For our purposes, probably the most convenient formulation of~\shortMethodName{}, allowing a unified treatment for the case $\mu=0$, is as presented in Algorithm~\ref{alg:item}.
\begin{algorithm}
	\begin{itemize}
		\item[] {\bf{}Input:} $f\in\mathcal{F}_{\mu,L}$ with $0\leq \mu<L<\infty$, initial guess $x_0\in\mathbb{R}^d$
		\item[] {\bf{}Initialization:} $y_{-1}=z_0=x_0$, $A_0=0$, $\cond=\mu/L$
		\item[] {\bf{}For} $k=0,1,\hdots$
\begin{equation}\label{eq:OGM_sc}
\begin{aligned}
\text{Set } A_{k+1}&= \frac{(1+\cond) A_k+2 \left(1+\sqrt{(1+A_k) (1+\cond  A_k)}\right)}{(1-\cond )^2}\\
\beta_k&=\frac{A_k}{(1-\cond)A_{k+1}}\text{, and }  \delta_k=\frac1{2}\frac{(1-\cond )^2 A_{k+1}-(1+\cond ) A_k}{1+\cond+\cond  A_k}\\[8pt]
{y_{k}}&=(1-\beta_k) z_k+\beta_k x_{k} \\
x_{k+1}&={y_{k}-\frac1L \nabla f(y_{k})}\\
z_{k+1}&={(1-\cond \delta_k)z_k+\cond \delta_k y_{k}-\frac{\delta_k}L  \nabla f(y_{k})}.
\end{aligned}
\end{equation}
\end{itemize}  
\caption{\fullMethodName{ }(\shortMethodName)}\label{alg:item}
\end{algorithm}

The following theorem states the main results concerning Algorithm~\ref{alg:item}: firstly, a bound on $\|z_N-x_\star\|^2$, and secondly, a bound involving function values, which is more relevant as $\mu\rightarrow0$. A proof for this theorem is provided in the next section.

\begin{theorem}\label{corr:finalbound}
Let $f\in\mathcal{F}_{\mu,L}$ and denote $\cond=\mu/L$. For any $x_0=z_0\in\mathbb{R}^d$ and $N\in\mathbb{N}$ with $N\geq 1$, the iterates of~\eqref{eq:OGM_sc} satisfy
\begin{equation*}
\begin{aligned}
\|z_N-x_\star\|^2 &\leq \frac{1}{1+\cond A_N}\|z_0-x_\star\|^2\leq \tfrac{(1-\sqrt{\cond})^{2N}}{(1-\sqrt{\cond})^{2N}+\cond}\|z_0-x_\star\|^2,\\
\psi_N &\leq \frac{L}{(1-\cond)A_{N+1}}\|z_0-x_\star\|^2\leq \min\left\{ (1-\sqrt{\cond})^{2(N+1)},\frac{1}{(N+1)^2}\right\}\frac{L}{(1-\cond)}\|z_0-x_\star\|^2,
\end{aligned}
\end{equation*}
with $\psi_N={f(y_{N})-f_\star-\tfrac{1}{2L}\|\nabla f(y_{N})\|^2-\tfrac{\mu}{2(1-\mu/L)}\|y_{N}-\tfrac1L \nabla f(y_{N})-x_\star\|^2}\geq 0$.
\end{theorem}
The quantity $\psi_N$ defined above is related to a potential (or Lyapunov) function that turns out to be key to the analysis of the method as provided in the next section. Although $\psi_N$ might appear as slightly unnatural, it should be interpreted in light of Theorem~\ref{thm:interp} with $x=x_\star$ and $y=y_N$, which ensures that $\psi_N\geq 0$. In the special case of $\mu=0$, $\psi_N$ corresponds to $f(y_N)-f_\star-\tfrac{1}{2L}\|\nabla f(y_N)\|^2$, thus applying the standard \emph{descent lemma} stating that $f(x_{N+1})\leq f(y_N)-\tfrac1{2L}\|\nabla f(y_N)\|^2$, we end up with a classical guarantee of type $f(x_{N+1})-f_\star\leq \psi_N\leq \tfrac{L}{(N+1)^2}\|z_0-x_\star\|^2$.

Note that as a result of Theorem~\ref{corr:finalbound}, the three sequences generated by Algorithm~\ref{alg:item} are all valid approximations of $x_\star$ in the following senses: (i) $z_N$ converges to the optimal solution in terms of distance to the solution, (ii) $y_N$ has a guarantee of having a small corresponding $\psi_N$, and (iii) $x_N$ is obtained from a gradient step from $y_{N-1}$, corresponds to having a guarantee on $f(x_N)-f_\star$ when $\mu=0$. 

\subsection{Worst-case analysis}\label{sec:analysis}
For performing the analysis, we use a potential function argument (see e.g., the nice review by~\citet{bansal2019potential}) similar to those used for standard accelerated methods~\citep{nest83,beck2009fast}. We show that for all ${y_{k-1}},z_k\in\mathbb{R}^d$ and $A_k\geq 0$, the function
\begin{equation}\label{eq:potential}
\begin{aligned}
\phi_k=&(1-\cond)A_k \psi_{k-1}+(L+\mu A_k)\|z_k-x_\star\|^2\\=&(1-\cond)A_{k}{\left[f(y_{k-1})-f_\star-\tfrac{1}{2L}\|\nabla f(y_{k-1})\|^{{2}}-\tfrac{\mu}{2(1-\mu/L)}\|y_{k-1}-\tfrac1L \nabla f(y_{k-1})-x_\star\|^2\right]}\\&+\left(
L+\mu A_{k}\right)\|z_{k}-x_\star\|^2
\end{aligned}
\end{equation}
satisfies $\phi_{k+1}\leq\phi_k$ when $z_{k+1}$, $y_{k}$ and $A_{k+1}$ are generated according to~\eqref{eq:OGM_sc}. 
\begin{lemma}\label{thm:potential}
Let $f\in\mathcal{F}_{\mu,L}$ and $k\geq 0$. For any $y_{k-1},z_k\in\mathbb{R}^d$ and $A_k\geq 0$, two consecutive iterations of~\eqref{eq:OGM_sc} satisfy
\[ \phi_{k+1}\leq \phi_k\]
with $A_{k+1}$ being defined as in~\eqref{eq:OGM_sc}.
\end{lemma}
\begin{proof}
We perform a weighted sum of two inequalities due to Theorem~\ref{thm:interp}:
\begin{itemize}
    \item smoothness and strong convexity of $f$ between $x_\star$ and {$y_{k}$} with weight $\lambda_1=(1-\cond)(A_{k+1}-A_k)$
    \begin{equation*}
    \begin{aligned}
    f_\star \geq& {f(y_{k})+\langle \nabla f(y_{k});x_\star-y_{k}\rangle+\tfrac1{2L}\|\nabla f(y_{k})\|^{{2}}+\tfrac{\mu}{2(1-\cond)}\|y_{k}-x_\star-\tfrac1L \nabla f(y_{k})\|^2},
    \end{aligned}
    \end{equation*}
    \item smoothness and strong convexity between ${y_{k-1}}$ and ${y_{k}}$ with weight $\lambda_2=(1-\cond)A_k$
    \begin{equation*}
    \begin{aligned}
    {f(y_{k-1})} \geq& {f(y_{k})+\langle \nabla f(y_{k});y_{k-1}-y_{k}\rangle+\tfrac1{2L}\|\nabla f(y_{k})-\nabla f(y_{k-1})\|^2}\\&{+\tfrac{\mu}{2(1-\cond)}\|y_{k}-y_{k-1}-\tfrac1L (\nabla f(y_{k})-\nabla f(y_{k-1}))\|^2}.
    \end{aligned}
    \end{equation*}
\end{itemize}
Summing up and reorganizing those two inequalities (without substituting $A_{k+1}$ by its definition for now), we arrive to the following inequality
\begin{equation*}
\begin{aligned}
0\geq \lambda_1 &\bigg[{f(y_{k})-f_\star+\langle \nabla f(y_{k});x_\star-y_{k}\rangle+\tfrac1{2L}\|\nabla f(y_{k})\|^2+\tfrac{\mu}{2(1-\cond)}\|y_{k}-x_\star-\tfrac1L \nabla f(y_{k})\|^2}\bigg]\\
+\lambda_2 &\bigg[ {f(y_{k}){-f(y_{k-1})}+\langle \nabla f(y_{k});y_{k-1}-y_{k}\rangle+\tfrac1{2L}\|\nabla f(y_{k})-\nabla f(y_{k-1})\|^2}\\
&\quad+ {\tfrac{\mu}{2(1-\cond)}\|y_{k}-y_{k-1}-\tfrac1L (\nabla f(y_{k})-\nabla f(y_{k-1}))\|^2}\bigg].
\end{aligned}
\end{equation*}
Substituting
\begin{equation*}
\begin{aligned}
{y_{k}}&={(1-\beta_k) z_k+\beta_k \left(y_{k-1}-\frac1L \nabla f(y_{k-1})\right)} \\
z_{k+1}&={(1-\cond \delta_k)z_k+\cond \delta_k y_{k}-\frac{\delta_k}L \nabla f(y_{k})},
\end{aligned}
\end{equation*}
(note that this substitution is also valid when $k=0$ as $\beta_k=0$ in this case, and hence $y_0=z_0$) the weighted sum can be reformulated exactly as (this can be verified by expanding both expressions and matching them on a term by term basis\footnote{The puzzled reader can verify this using basic symbolic computations. We provide a notebook for verifying the equivalence of the expressions in Section~\ref{s:ccl}.}):
\begin{equation*}
\begin{aligned}
\phi_{k+1}
\leq& \phi_k-L K_1 P(A_{k+1},A_k) \|z_k-x_\star\|^2\\
&+\tfrac1{4L}K_2 P(A_{k+1},A_k)\| (1-\cond) A_{k+1}{\nabla f(y_{k})}-\mu A_k \left({y_{k-1}}-x_\star-\tfrac1L {\nabla f(y_{k-1})}\right)+K_3\mu (z_k-x_\star)\|^2
\end{aligned}
\end{equation*}
with three constants (well defined given that $0\leq\mu<L<\infty$ and $A_k,A_{k+1}\geq 0$)
\begin{equation*}
\begin{aligned}
K_1&=\frac{ \cond ^2 }{(1+\cond)^2+(1-\cond)^2\cond A_{k+1}}\\
K_2&=\frac{(1+\cond)^2+(1-\cond )^2 \cond  A_{k+1}}{(1-\cond)^2\left(1+\cond+\cond  A_k\right)^2 A_{k+1}^2}\\
K_3&=(1+\cond)\frac{( 1+\cond) A_k- (1-\cond)(2+\cond A_k)A_{k+1}}{(1+\cond)^2+(1-\cond )^2 \cond  A_{k+1}},
\end{aligned}
\end{equation*}
as well as
\[P(x,y)=(y-(1-\cond )x)^2-4 x (1+\cond  y).\] 
For obtaining the desired potential inequality, it remains to remark that $A_{k+1}$ corresponds to the largest solution of $P(x,A_k)=0$. That is, the weighted sum can be reorganized exactly as  $\phi_{k+1}\leq \phi_k$, reaching the desired claim.\qed
\end{proof}
We are now equipped for proving our main result, presented in Theorem~\ref{corr:finalbound}.
\begin{proof}[Theorem~\ref{corr:finalbound}]
From Lemma~\ref{thm:potential}, we get
\[ \phi_N\leq\phi_{N-1}\leq \hdots\leq \phi_0=L\|z_0-x_\star\|^2.\] 
From Theorem~\ref{thm:interp} (evaluated at $x\leftarrow x_\star$, and $y\leftarrow y_{k-1}$), we have that $(L+\mu A_N)\|z_N-x_\star\|^2\leq \phi_N$, reaching
\[ \|z_N-x_\star\|^2 \leq \frac{\phi_0}{(L+\mu A_N)}=\frac{1}{1+\cond A_N}\|z_0-x_\star\|^2.\]
Similarly, we have that $(1-\cond)A_{N+1}\psi_{N}\leq \phi_{N+1}\leq \phi_N$ and hence
\[ \psi_N \leq \frac{\phi_0}{(1-\cond)A_{N+1}}=\frac{{L}}{(1-\cond)A_{N+1}}\|z_0-x_\star\|^2.\]
For reaching the claims, it is therefore sufficient to characterize the growth rate of $\{A_k\}$. Because finding a closed-form expression for $\{A_k\}$ appears to be out of reach, we consider the classical two scenarios for bounding its growth rate. First, when $\mu=0$,
\[A_{k+1}=2+A_k+2\sqrt{1+A_k}\geq 2+A_k+2\sqrt{A_k}\geq (1+\sqrt{A_k})^2,\]
reaching $\sqrt{A_{k+1}}\geq 1+\sqrt{A_k}$ and hence $\sqrt{A_k}\geq k$ and $A_k\geq k^2$. Second, when $\mu>0$, one also has
\[A_{k+1}=\frac{(1+\cond) A_k+2 \left(1+\sqrt{(1+A_k) (1+\cond  A_k)}\right)}{(1-\cond )^2}\geq \frac{(1+\cond)A_k+2 \sqrt{\cond  A_k^2}}{(1-\cond)^2}=\frac{A_k}{(1-\sqrt{\cond})^2}.\] This last bound, together with $A_1=\tfrac{4}{(1-\cond)^2}=\tfrac{4}{(1+\sqrt{\cond})^2(1-\sqrt{\cond})^2}\geq (1-\sqrt{\cond})^{-2}$, allows reaching the target $A_N\geq (1-\sqrt{\cond})^{-2N}$, thereby concluding the proof.\qed
\end{proof}

\subsection{Limit cases}\label{s:limits}
In this section, we inspect two limit cases of~\shortMethodName. First, when $\mu=0$, \shortMethodName{} can be compared to the Optimized Gradient Method of~\citet{kim2015optimized}. In their notations, we denote by $\theta_k^2=\tfrac{A_{k+1}}{4}$, a sequence that can alternatively be defined recursively as $\theta_0=1$ and $\theta_{k+1}=\frac{1+\sqrt{4\theta_k^2+1}}{2}$. In this setting, the parameters correspond to $\beta_k=\frac{A_k}{A_{k+1}}$, and $\delta_k=\frac{A_{k+1}-A_k}{2}$, and we recover, using~\citet{kim2015optimized}'s notations (using the identity $\theta_k^2=\theta_{k-1}^2+\theta_k$)
    \begin{equation*}
    \begin{aligned}
    {y_{k}}&=\tfrac{\theta_{k}-1}{\theta_{k}}x_k+\tfrac{1}{\theta_{k}}z_k\\
    x_{k+1}&={y_{k}-\tfrac1L \nabla f(y_{k})}\\
    z_{k+1}&={z_k-\tfrac{2}{L}\theta_k \nabla f(y_{k})}.
    \end{aligned}
    \end{equation*}
Note though that~\citet{kim2015optimized} uses a ``last iteration adjustment'' by setting $\theta_N=\tfrac{1+\sqrt{8\theta_{N-1}^2+1}}{2}$. This adjustment is not needed for the purpose of obtaining the optimal bound on $\|z_N-x_\star\|$, and a detailed treatment can be found in~\citep[Section 4.3.1]{dAspremontAccelerated}.

Second, when $\mu>0$ and $k\rightarrow\infty$, one can explicitly compute the limits of the algorithmic parameters
\begin{equation*}
\begin{aligned}
&\lim_{k\rightarrow\infty}\frac{A_{k}}{A_{k+1}}=\lim_{A_k\rightarrow\infty}\frac{(1-\cond)^2A_k}{(1+\cond)A_k+2+2\sqrt{1+(1+\cond)A_k+\cond A_k^2}} =\frac{(1-\cond)^2}{(1+\sqrt{\cond})^2}=(1-\sqrt{\cond})^2\\
&\lim_{k\rightarrow\infty}\beta_k=\lim_{k\rightarrow\infty}\frac{A_{k}}{(1-\cond)A_{k+1}}=\frac{1-\sqrt{\cond}}{1+\sqrt{\cond}}\\
&\lim_{k\rightarrow\infty}\delta_k=\lim_{k\rightarrow\infty}\frac1{2}\frac{(1-\cond )^2 A_{k+1}-(1+\cond ) A_k}{1+\cond+\cond  A_k}=\frac12 \frac{(1-\cond)^2-(1+\cond)(1-\sqrt{\cond})^2}{\cond(1-\sqrt{\cond})^2} =\sqrt{\frac1{\cond}},
\end{aligned}
\end{equation*}
reaching
\begin{equation*}
\begin{aligned}
{y_{k}}&=\frac{1-\sqrt{\cond}}{1+\sqrt{\cond}} {\left(y_{k-1}-\tfrac1L \nabla f(y_{k-1})\right)}+ \left(1-\frac{1-\sqrt{\cond}}{1+\sqrt{\cond}}\right) z_k \\
z_{k+1}&=\sqrt{\cond} \big({y_{k}-\tfrac1\mu  \nabla f(y_{k})}\big)+(1-\sqrt{\cond})z_k,
\end{aligned}
\end{equation*}
which is the Triple Momentum Method~\citep{van2017fastest} and its convergence rate $(1-\sqrt{q})^2$. 

For those two limit cases, the analysis from Section~\ref{sec:analysis} can be simplified accordingly. For the OGM, this leads to the same potential as that provided in e.g.,~\citep[Theorem 11]{taylor19bach}), or~\citep[Section 2]{park2021factor}. For the TMM, this allows recovering the known Lyapunov function from e.g.,~\citep[Inequality (10)]{cyrus2018robust}.

\subsection{Lower bound and matching examples}\label{s:LB}

In this section, we show the correspondence with the lower bound from~\citep{drori2021exact} and provide two very simple one-dimensional examples on which the method achieves its worst-case.

First, the lower bound from~\citep[Corollary 4]{drori2021exact} states that for any black-box first-order, there exists $f\in\mathcal{F}_{\mu,L}$ such that
\[\|x_N-x_\star\|^2\geq \frac{\sz_N^2}{\cond}\|x_0-x_\star\|^2,\]
where $x_\star=\mathrm{argmin}_x f(x)$, $x_N$ is the output of the black-box first-order method under consideration, and where the sequence $\{\sz_i\}$ is defined recursively as $\sz_0=\sqrt{\cond}$ and
\[ \sz_{k+1}= \frac{1-\sqrt{\cond - (1 - \cond) \sz_k^2}}{1+ \sz_k^2} \sz_k.\]
Let us show that it matches the upper bound provided by Theorem~\ref{corr:finalbound}. One can verify the identity
\[ \sz_k^2=\frac{\cond}{1+\cond A_k},\]
by observing that it holds for $k=0$ with $A_0=0$ then using an inductive argument. That is, assuming $\sz_k=\sqrt{\frac{{\cond}}{1+\cond A_k}}$, it is relatively simple to establish that
\[\sz_{k+1}=\frac{1-\sqrt{\cond - (1 - \cond) \left(\frac{{\cond}}{1+\cond A_k}\right)}}{1+ \left(\frac{{\cond}}{1+\cond A_k}\right)} \sqrt{\frac{{\cond}}{1+\cond A_k}}=\sqrt{\frac{{\cond}}{1+\cond A_{k+1}}}.\]

Because the lower bound from~\citep{drori2021exact} and the upper bound from Theorem~\ref{corr:finalbound} match, it is clear that the worst-case guarantee of \shortMethodName{} cannot be improved. 

\medskip

The lower bound proof from~\citep{drori2021exact} is constructive in the sense that it exhibits a ``worst function in the world'' on which any first-order method cannot attain a worst-case guarantee better than the one stated above. Clearly, such a function would naturally attain the worst-case behavior of \shortMethodName{}, however, this function is rather complex and it is the purpose of the following paragraphs to show that the worst-case behavior of \shortMethodName{} is also attained on very simple functions. In particular, the worst-case is achieved on the two base quadratic functions
\[ f_L(x)=\frac{L}{2}|x|^2,\quad f_{\mu}{(x)}=\frac{\mu}{2}|x|^2,\]
i.e., the guarantee $\|z_N-x_\star\|^2\leq\frac{\|x_0-x_\star\|^2}{1+\cond A_N}$ holds with equality on both $f_L({\cdot})$ and~$f_\mu({\cdot})$.
{\begin{lemma} Let  $0< \mu< L< \infty$, and $f_L,f_{\mu}\in\mathcal{F}_{\mu,L}(\mathbb{R})$ with $f_{\mu}(x)=\tfrac{\mu}{2}x^2$ and $f_L(x)=\tfrac{L}{2}x^2$. The iterates of \shortMethodName~\eqref{eq:OGM_sc} satisfy
\[z_{k}^2=\frac{z_0^2}{1+\cond A_k}\]
when applied to either $f_{\mu}$ or $f_L$.
\end{lemma}
\begin{proof} We proceed by recurrence. It is clear that $z_0^2=\tfrac{z_0^2}{1+\cond A_0}$ (recall $A_0=0$), which establishes the base recurrence case.

(i) Let us start with $f_L$. It is clear from explicit computations that for all $y_k\in\mathbb{R}$, $x_{k+1}=y_k-\tfrac1{L}\nabla f_{L}(y_k)=x_\star=0$. Therefore, we have $y_k=\tfrac1L\nabla f_L(y_k)$ along with
\[ y_{k}=(1-\beta_k)z_k\]
(this also trivially holds for $k=0$, as in this case $\beta_k=0$ and $y_0=z_0=x_0$), and therefore
\[z_{k+1}=z_k+\cond \delta_k (y_k-z_k)-\delta_k y_k= ((1-\cond) \beta_k \delta_k-\delta_k+1)z_k.\]
Substituting the expressions of $\beta_k$, $\delta_k$, $A_{k+1}$, and $z_k^2=\tfrac{z_0^2}{1+\cond A_k}$ in this equality (squared) leads to
\[z_{k+1}^2= \frac{\left(1+\cond  A_k-\cond  \sqrt{(1+A_k) (1+\cond  A_k)}\right)^2}{(1+\cond  A_k) (1+\cond+\cond  A_k)^2} z_0^2=\frac{z_0^2}{1+\cond A_{k+1}}, \]
where the last equality can be verified by basic algebra.

(ii) We proceed with $f_{\mu}$. In this case, for all $y_k\in\mathbb{R}$ we have $y_k-\tfrac1{\mu}\nabla f_{\mu} (y_k)=x_\star=0$. Therefore,
\[z_{k+1}=(1-\cond \delta_k)z_k.\]
Substituting the expression of $\delta_k$ and the recurrence hypothesis $z_k^2=\tfrac{z_0^2}{1+\cond A_k}$ we arrive to the same expression as before
\[z_{k+1}^2= \frac{\left(1+\cond  A_k-\cond  \sqrt{(1+A_k) (1+\cond  A_k)}\right)^2}{(1+\cond  A_k) (1+\cond+\cond  A_k)^2} z_0^2=\frac{z_0^2}{1+\cond A_{k+1}}, \]
reaching the desired claim.\qed
\end{proof}}


\section{A constructive approach to \shortMethodName}\label{s:optimized_steps}
The intent of this section is to provide a constructive procedure for obtaining the \fullMethodName, as well as other similar methods designed based on alternate optimality criteria.

The construction is based on the performance estimation methodology introduced in \citep{drori2013performance,taylor2017smooth}, where the main idea is to cast the theoretical worst-case performance of a generic first-order method as an optimization program over all possible problem instances.
Once such a program has been devised, it can then be manipulated using standard techniques, and in particular, this allows us to state the problem of finding the ``best'' first-order method as a minimax problem. Although this minimax problem appears at first to be hard, we show that a tractable relaxation of it can be devised, and that~\shortMethodName{} can be obtained as an analytical solution to that problem.
We would like to emphasize that although \shortMethodName{} was discovered using the technique described below, its proof, as provided above, is independent of the following.

As a starting point, consider the class of black-box first-order methods gathering information about the objective function $f$ only by evaluating an \emph{oracle} $\mathcal{O}_f(x)=(f(x),\nabla f(x))$. We describe such a black-box method $M$ as a set of rules $\{M_1,M_2,\hdots,M_N\}$ for forming its iterates, which we denote by $w_k$ for avoiding confusions with any of the sequences defined by \shortMethodName, as 
\begin{equation*}
\begin{aligned}
w_1&=M_1(w_0,\mathcal{O}_f(w_0))\\
w_2&=M_2(w_0,\mathcal{O}_f(w_0),\mathcal{O}_f(w_1))\\
&\vdots\\
w_N&=M_N(w_0,\mathcal{O}_f(w_0),\mathcal{O}_f(w_1),\hdots,\mathcal{O}_f(w_{N-1})),
\end{aligned}
\end{equation*}
and we denote by $\mathcal{M}_N$ the set of black-box first-order methods that perform $N$ gradient evaluations. Furthermore, we call the \emph{efficiency estimate} of a method $M$ the following quantity 
\begin{equation}\label{eq:eff_estimate}
\begin{aligned}
W_{\mu,L}(M)=\sup_{f\in\mathcal{F}_{\mu,L}} \bigg\{\frac{\|w_N-w_\star\|^2}{\|w_0-w_\star\|^2}:\,& \text{for any sequence $w_1,\dots,w_N$ generated by $M$ on $f$},\\&\text{initiated at some $w_0$, and $w_\star\in\mathrm{argmin}_wf(w)$}\bigg\},
\end{aligned}
\end{equation}
which correspond to the worst-case performance of $M$ on the class $\mathcal{F}_{\mu,L}$ for the criterion $\tfrac{\|w_N-w_\star\|^2}{\|w_0-w_\star\|^2}$. A~direct consequence of Theorem~\ref{corr:finalbound} and the lower complexity bound discussed in Section~\ref{s:LB} is that \shortMethodName{} belong to the class of black-box first-order methods with optimal performances with respect to $W_{\mu,L}(M)$ with $M\in\mathcal{M}_N$. \shortMethodName{} is therefore a solution to
\begin{equation}\label{eq:blackbox_optim}
\begin{aligned}
\min_{M\in\mathcal{M}_N}\,W_{\mu,L}(M).
\end{aligned}
\end{equation}
{Although this minimax problem appears to be hard to solve directly,}
we illustrate below that it can be approached using semidefinite programming. 

In a nutshell, we consider two simplified upper bounds to this minimax problem. First, we consider a subclass of black-box first-order methods, referred to as \emph{fixed-step first-order methods}. Those are first-order methods that are described by a set of fixed coefficients $\{h_{i,j}\}$, and whose formal description is provided below. Second, given a fixed-step first-order method $M$, the idea is to develop a tractable upper bound on the efficiency estimate of $M$, written $\mathrm{UB}_{\mu,L}(M)$ and such that $\mathrm{UB}_{\mu,L}(M)\geq W_{\mu,L}(M)$. 
After that, we show that minimizing this upper bound over $M$ is also tractable. That is, we can solve $\min_{\{h_{i,j}\}}\mathrm{UB}_{\mu,L}(M)$ to obtain the \fullMethodName{} as a solution.

As a comparison, let us mention that the Optimized Gradient Method~\citep{drori2013performance,kim2015optimized} was obtained through similar steps for the objective $({f(w_N)-f_\star})/{\|w_0-w_\star\|^2}$ when~$\mu=0$. Note, however, that a straightforward application of the technique presented in~\citep{drori2013performance,kim2015optimized} does not yield tractable problems in the strongly convex case.

More precisely, we proceed as follows:
\begin{itemize}
    \item In Section~\ref{s:FO_class}, we describe the class of fixed-step first-order methods. This class of methods is somewhat natural and contains classical numerical methods such as gradient, heavy-ball, and accelerated gradient methods, but excludes adaptive methods. For this class of methods, it is known that $W_{\mu,L}(M)$ can be formulated as a convex semidefinite program (see e.g.,~\citep[Theorem 6]{taylor2017smooth}). However, when it comes to optimizing over step size parameters, this formulation leads to a bilinear/quadratic problem which we do not know how to solve directly.
    \item In Section~\ref{s:PEP} and~\ref{s:PEP_SDP}, we provide an equivalent reparametrization of the class of fixed-step first-order methods, allowing to reach an alternate semidefinite formulation for $W_{\mu,L}(M)$ with simpler structure. We further detail a tractable upper bound $\mathrm{UB}_{\mu,L}(M)$ which is more convenient for optimizing over the method's parameters.
    \item In Section~\ref{s:approx_minimax}, we show how to render $\min_{\{h_{i,j}\}}\mathrm{UB}_{\mu,L}(M)$ tractable, yielding the \fullMethodName{} as a solution.
\end{itemize}

We complement those developments by numerically designing first-order methods for alternate design criterion that include $(f(w_N)-f_\star)/\|{w_0}-w_\star\|^2$. For doing that, the developments of this section have to be slightly adapted (see Appendix~\ref{a:func_values}). The  numerical results are provided in Appendix~\ref{a:numerics}, and source code for reproducing the results is provided in Section~\ref{s:ccl}.

\subsection{Fixed-step first-order methods}\label{s:FO_class}
In this section, we introduce a subclass of black-box first-order methods described by a set of fixed coefficients. This parametric subset of $\mathcal{M}_N$ allows for more convenient formulations of optimization problems over the class of methods, such as the minimax problem~\eqref{eq:blackbox_optim}.

We start with the following ``{natural}'' description of the class of methods of interest, then introduce an alternate parametrization which is more convenient for the step size optimization procedure of the following sections.
\begin{definition}\label{def:FS} A black box first-order method is called a \emph{fixed-step first-order method} if there exists a set $\{h_{i,j}\}\subset \mathbb{R}$ 
such that the method admits the following description
\begin{equation}\label{eq:method}
\begin{aligned}
w_1&=w_0-\tfrac{h_{1,0}}{L} \nabla f(w_0)\\
w_2&=w_1-\tfrac{h_{2,0}}{L} \nabla f(w_0)-\tfrac{h_{2,1}}{L}\nabla f(w_1)\\
&\vdots\\
w_N&=w_{N-1}-\sum_{i=0}^{N-1}\tfrac{h_{N,i}}{L}\nabla f(w_i),
\end{aligned}
\end{equation}
for any function $f$.
\end{definition}
For fixed-step first-order methods $M$ described by a set of normalized coefficients $\{h_{i,j}\}$, it is shown in~\citep[Theorem 6]{taylor2017smooth} that $W_{\mu,L}(M)$ can be formulated as a convex semidefinite program~(SDP). Given such an SDP formulation, our goal is to solve
\[  \min_{\{h_{i,j}\}}W_{\mu,L}(M),\]
which is a bilinear/quadratic problem, due to the structure of the SDP formulation of $W_{\mu,L}(M)$ in~\citep[Theorem 6]{taylor2017smooth}. Such problems are nonconvex and NP-hard in general~\citep{toker1995np},
nevertheless, by performing reparametrization followed by relaxation and linearization steps, as shown in the following sections, it is possible to attain a tractable relaxation of the problem.

\subsection{A reparametrization of fixed-step first-order methods}
In what follows, we restrict ourselves to these fixed-step first-order methods, which we will reparameterize in a slightly different, but equivalent, fashion.  Informally, the alternate parameterization allows formulating the maximization problem arising in the efficiency estimate $W_{\mu,L}(M)$ (see~\eqref{eq:eff_estimate}) in a more convenient way than that of~\citep[Theorem 6]{taylor2017smooth} for our purposes. Indeed, the new formulation presented in the next sections allows obtaining a problem that is ``only'' bilinear in terms of the method parameters and of some multipliers $\lambda_{i,j}$'s. Those problems are still NP-hard in general~\citep{toker1995np}, however, in this case this simplification will enable us to optimize over the method parameters, a simplification that appears to be hard to reach with previous formulations.

In order to proceed, we express first-order methods for minimizing $f$ as acting instead on a function~$\tilde f$, using
\[ \tilde{f}(x):=f(x)-\tfrac{\mu}{2}\|x-w_\star\|^2, \]
where $w_\star$ is a minimizer of both $f$ and $\tilde{f}$. It is known (see e.g.~\citep{nest-book-04}) that  $f\in\mathcal{F}_{\mu,L}$ if and only if $\tilde f\in\mathcal{F}_{0,L-\mu}$. Then, one can express~\eqref{eq:method} in terms of evaluations of the gradient of $\tilde{f}$, instead of that of $f$. Concretely, we reformulate~\eqref{eq:method} in terms of some coefficients $\{\alpha_{i,j}\}$ as follows
\begin{equation}\label{eq:method_reparam}
\begin{aligned}
w_1-w_\star&=(w_0-w_\star)(1-\tfrac{\mu}{L}\alpha_{1,0})-\tfrac{\alpha_{1,0}}{L} \nabla \tilde f(w_0)\\
w_2-w_\star&=(w_0-w_\star)(1-\tfrac{\mu}{L}(\alpha_{2,0}+\alpha_{2,1}))-\tfrac{\alpha_{2,0}}{L} \nabla \tilde f(w_0)-\tfrac{\alpha_{2,1}}{L}\nabla \tilde f(w_1)\\
&\vdots\\
w_N-w_\star&=(w_{0}-w_\star)\left(1-\tfrac{\mu}{L}\sum_{i=0}^{N-1}\alpha_{N,i}\right)-\sum_{i=0}^{N-1}\tfrac{\alpha_{N,i}}{L}\nabla \tilde f(w_i).
\end{aligned}
\end{equation}
One can show that there is a bijection between representations~\eqref{eq:method} and~\eqref{eq:method_reparam}. Therefore, the problem of designing an optimal method in the form~\eqref{eq:method} is equivalent to that of devising an optimal method in the form~\eqref{eq:method_reparam}. This is formalized by the following lemma.
\begin{lemma}\label{lem:eq_param} Let $N\in\mathbb{N}$ and a first-order method $M\in\mathcal{M}_N$. The following statements are equivalent.
\begin{itemize}
    \item There exists a set $\{h_{i,j}\}_{i=1,\hdots,N;j=0,\hdots,i-1}$ such that for any $f\in\mathcal{F}_{\mu,L}$ and $w_0\in\mathbb{R}^d$ the sequence $\{w_k\}_{k=0,\hdots,N}\subset\mathbb{R}^d$ generated by $M$ satisfies~\eqref{eq:method} (i.e., $M$ is a fixed-step first-order method).
    \item There exists a set $\{\alpha_{i,j}\}_{i=1,\hdots,N;j=0,\hdots,i-1}$ such that for any $f\in\mathcal{F}_{\mu,L}$ and $w_0\in\mathbb{R}^d$ the sequence $\{w_k\}_{k=0,\hdots,N}\subset\mathbb{R}^d$ generated by $M$ satisfies~\eqref{eq:method_reparam}.
\end{itemize}
\end{lemma}
\begin{proof} The proof follows from a short recurrence argument (provided in Appendix~\ref{s:reparam}) for showing that the two representations are isomorphic, and that they are linked through the following triangular system of equations
\begin{equation}\label{eq:equiv_param}
\begin{aligned}
\alpha_{k+1,i}=\left\{\begin{array}{ll}
h_{k+1,k}     & \text{if } i=k \\
h_{k+1,i}+\alpha_{k,i}-\tfrac{\mu}{L}\sum_{j=i+1}^k h_{k+1,j}\alpha_{j,i}\quad    & \text{if } 0\leq i <k.
\end{array}\right.
\end{aligned}
\end{equation}
Therefore, although we use~\eqref{eq:method_reparam} in the following sections, any method formulated in terms of $\{\alpha_{i,j}\}$ can be converted to the more natural $\{h_{i,j}\}$ notation, and reciprocally. \qed
\end{proof}

In the next section, we develop an upper bound on $W_{\mu,L}(M)$ of a form similar to that of~\citep[Theorem 6]{taylor2017smooth}, but which is linear in $\{\alpha_{i,j}\}$, instead of quadratic in $\{h_{i,j}\}$.
\subsection{A performance estimation problem and its relaxation}\label{s:PEP}

The goal of this section is to construct an upper bound on~\eqref{eq:eff_estimate} that can be computed efficiently. The reformulation and relaxation techniques used for obtaining the upper bound are not new and rely on the same steps as those taken in~\citep{taylor2017smooth} (so readers familiar with such procedures can safely fly over the section). We provide details which allows optimizing over the step sizes afterwards. Let us start by rephrasing~\eqref{eq:eff_estimate} as
\begin{equation*}
\begin{aligned}
W_{\mu,L}(M)=\sup_{\substack{\tilde f,d\in\mathbb{N}\\ \{w_i\}_{i\in I}\subset\mathbb{R}^d}}\,& \frac{\|w_{N}-w_\star\|^2}{\|w_0-w_\star\|^2}\\ \text{s.t. } & w_k \text{ generated by~\eqref{eq:method_reparam} applied on $\tilde{f}$, and initiated at $w_0$,}\\
&\tilde f\in\mathcal{F}_{0,L-\mu}(\mathbb{R}^d),\\
&\,w_\star\in\mathrm{argmin}_w \tilde f(w),
\end{aligned}
\end{equation*}
where we used an index set $I=\{\star,0,\hdots,N\}$. Note the maximization over $d$, which aims at obtaining dimension-independent guarantees. For such problems, it is known that the supremum is attained (see e.g.,~\citep[Proposition 1]{taylor2017smooth}), and we therefore use ``$\max$'' instead of ``$\sup$'' in what follows.

As a first step towards an ``efficient'' upper bound, we reformulate~\eqref{eq:eff_estimate} using an {\emph{extension}} (or interpolation) argument. That is, the previous maximization problem can be restated using an existence argument for replacing the function by a finite set of samples. In other words, we optimize over the oracle's responses while keeping the responses consistent with assumptions on $f$
\begin{equation*}
\begin{aligned}
\max_{\substack{\{(w_i,g_i,f_i)\}_{i\in I}\subset\mathbb{R}^d\times\mathbb{R}^d\times\mathbb{R}\\d\in\mathbb{N}}} &\frac{\|w_{N}-w_\star\|^2}{\|w_0-w_\star\|^2}\\
\text{s.t. }& w_{k} \text{ generated by~\eqref{eq:method_reparam} for } k=1,\hdots,N,\\
&\exists \tilde f\in\mathcal{F}_{0,L-\mu}(\mathbb{R}^d): g_i\in\partial \tilde f(w_i),\, f_i=\tilde f(w_i) \, \forall i\in I,\\
& g_\star=0.
\end{aligned}
\end{equation*} 
Using an homogeneity argument, one can reformulate this problem without the fractional objective. More precisely, for any feasible point $S=\{({w_i},g_i,f_i)\}_{i\in I}$ and any $\alpha>0$, the point $S'=\{({\alpha w_i}, \alpha  g_i,\alpha^2 f_i)\}_{i\in I}$ is also feasible while reaching the same objective value (this can be verified using the definition of $\mathcal{F}_{\mu,L}$). We can therefore arbitrarily fix the scale of the problem to $\|w_0-w_\star\|=1$, reaching the following problem with the same optimal value
\begin{equation*}
\begin{aligned}
\max_{\substack{\{(w_i,g_i,f_i)\}_{i\in I}\subset\mathbb{R}^d\times\mathbb{R}^d\times\mathbb{R}\\d\in\mathbb{N}}} &\|w_{N}-w_\star\|^2\\
\text{s.t. }& \|w_0-w_\star\|^2=1,\\
&w_{k} \text{ generated by~\eqref{eq:method_reparam} for } k=1,\hdots,N,\\&\exists \tilde f\in\mathcal{F}_{0,L-\mu}(\mathbb{R}^d): g_i\in\partial \tilde f(w_i),\, f_i=\tilde f(w_i) \, \forall i\in I,\\
& g_\star=0.
\end{aligned}
\end{equation*} 
It follows from Theorem~\ref{thm:interp2} that the previous problem can be reformulated exactly as
\begin{equation}\label{eq:pep}
\begin{aligned}
\max_{\substack{\{(w_i,g_i,f_i)\}_{i\in I}\subset\mathbb{R}^d\times\mathbb{R}^d\times\mathbb{R}\\d\in\mathbb{N}}} &\|w_{N}-w_\star\|^2\\
\text{s.t. }& \|w_0-w_\star\|^2=1,\\
&w_{k} \text{ generated by~\eqref{eq:method_reparam} for } k=1,\hdots,N,\\
&f_i\geq f_j+\langle g_j;x_i-x_j\rangle+\tfrac{1}{2(L-\mu)}\|g_i-g_j\|^2\,\,\, \forall i,j\in I,\\
& g_\star=0.
\end{aligned}
\end{equation} 
Whereas equivalence between the two previous problems might be regarded as technical, the fact~\eqref{eq:pep} produces upper bounds on~\eqref{eq:eff_estimate} is quite direct. Indeed, any $\tilde f\in\mathcal{F}_{0,L-\mu}$ satisfies the above inequalities, and hence any feasible point to~\eqref{eq:eff_estimate} can be converted to a feasible point to~\eqref{eq:pep} by sampling $\tilde f$.

In what follows, we use the following relaxation of~\eqref{eq:pep}, by incorporating only a specific subset of the previous quadratic inequalities, therefore forming an upper bound on the original problem. Many inequalities were removed because they introduce undesirable nonlinearities in the steps taken in the next sections. Perhaps luckily, this relaxation will turn out to be tight for evaluating $W_{\mu,L}(M)$ of~\shortMethodName{}.
\begin{align}
\max_{\substack{\{(w_i,g_i,f_i)\}_{i\in I}\subset\mathbb{R}^d\times\mathbb{R}^d\times\mathbb{R}\\d\in\mathbb{N}}}&\|w_N-w_\star\|^2&&\notag\\
\text{s.t. }& \|w_0-w_\star\|^2=1,\, g_\star=0,&&\label{eq:relaxed}\tag{R}\\
&w_{k} \text{ generated by~\eqref{eq:method_reparam} } &&\text{for }k=1,\hdots,N,\notag\\
& f_i\geq f_{i+1}+\langle g_{i+1};w_i-w_{i+1}\rangle +\tfrac1{2(L-\mu)}\|g_i-g_{i+1}\|^2 &&\text{for }i=0,\hdots,N-2,\notag\\
& f_\star\geq f_i+\langle g_i,w_\star-w_{i}\rangle +\tfrac1{2(L-\mu)}\|g_{i}\|^2 &&\text{for }i=0,\hdots,N-1,\notag\\
& f_{N-1}\geq f_\star+\tfrac1{2(L-\mu)}\|g_{N-1}\|^2. &&\notag
\end{align}
As shown in the next section, this problem is semidefinite-representable and we can thus use standard packages for approximating its solution numerically. Looking at the structure of these numerical solutions helped us to choose this particular relaxation.

The following lemma summarizes what we have obtained so far, that is,  $W_{\mu,L}(M)\leq\mathrm{val}\text{\eqref{eq:relaxed}}$, where $\mathrm{val}\text{\eqref{eq:relaxed}}$ denotes the optimal value of~\eqref{eq:relaxed}.
\begin{lemma}\label{lem:pep}Let $N\in\mathbb{N}$, $0\leq\mu<L<\infty$, and $M\in\mathcal{M}_N$ be a fixed-step first-order method~\eqref{eq:method_reparam} performing $N$ gradient evaluations and described by a set of coefficients $\{\alpha_{i,j}\}_{i,j}$. For any $d\in\mathbb{N}$, $f\in\mathcal{F}_{\mu,L}(\mathbb{R}^d)$, $w_\star\in\mathrm{argmin}_w f(w)$, initial guess $w_0\in\mathbb{R}^d$, and $w_N=M(w_0,f)$, it holds that
\[\|w_N-w_\star\|^2\leq \mathrm{val}\text{\eqref{eq:relaxed}}\, \|w_0-w_\star\|^2,\]
where $\mathrm{val}\text{\eqref{eq:relaxed}}$ denotes the optimal value of~\eqref{eq:relaxed}. 
\end{lemma}

\subsection{Tractable upper bounds using semidefinite programming (SDP)}\label{s:PEP_SDP}
We now show how to reach a standard SDP formulation for~\eqref{eq:relaxed}. One can reformulate the maximization problem~\eqref{eq:relaxed} in terms of the variables $(G,F)$ (after substituting $w_k$'s by their expressions) defined by
\begin{equation}\label{eq:gram_representation}
\begin{aligned}
G&=\begin{pmatrix} \|w_0-w_\star\|^2 & \langle g_0;w_0-w_\star\rangle & \langle g_1;w_0-w_\star\rangle &\hdots & \langle g_{N-1};w_0-w_\star\rangle\\
\langle g_0;w_0-w_\star\rangle & \|g_0\|^2 & \langle g_1;g_0\rangle & \hdots &\langle g_{N-1};g_0\rangle\\
\langle g_1;w_0-w_\star\rangle & \langle g_1;g_0\rangle & \|g_1\|^2 &\hdots &\langle g_{N-1};g_1\rangle\\
\vdots & \vdots & \vdots &\ddots & \vdots\\
\langle g_{N-1};w_0-w_\star\rangle & \langle g_{N-1};g_0\rangle & \langle g_{N-1};g_1\rangle &\hdots &\|g_{N-1}\|^2\\
\end{pmatrix} \succeq 0,\\
F&=\begin{pmatrix}f_0-f_\star \\ f_1-f_\star \\ \vdots \\ f_{N-1}-f_\star \end{pmatrix},
\end{aligned}
\end{equation}
Formally, let us introduce the following notations for picking elements in $G$ and $F$ and conveniently formulating the SDPs
\[ \bxx_0=e_1\in\mathbb{R}^{N+1},\quad \bgg_i=e_{i+2}\in\mathbb{R}^{N+1},\quad \bff_i=e_{i+1}\in\mathbb{R}^{N},\]
with $i=0,\hdots,N-1$ and $e_i$ being the unit vector whose $i$th component is equal to $1$. In addition, we also denote by
\[\bxx_k=\bxx_0\left(1-\tfrac{\mu}{L}\sum_{i=0}^{k-1}\alpha_{k,i}\right)-\sum_{i=0}^{k-1}\tfrac{\alpha_{k,i}}{L}\bgg_i,\]
for $i=0,\hdots,N$ (note that $\bxx_k$ is therefore linearly parameterized by $\{\alpha_{k,i}\}$). Those notations allow to express the objective and constraints of~\eqref{eq:relaxed} directly in terms of $G$ and $F$ using the following identities
\begin{equation*}
\begin{aligned}
&f_i-f_\star= F\bff_i && i=0,1,\hdots,N-1,\\
&\| g_i\|^2= \bgg_i^\top G\bgg_i && i=0,1,\hdots,N-1,\\
&\| w_i-w_\star\|^2= \bxx_i^\top G\bxx_i && i=0,1,\hdots,N,\\
&\langle g_i;w_j-w_\star\rangle=\bgg_i^\top G \bxx_j \quad && i=0,1,\hdots,N-1,\, j=0,1,\hdots,N.
\end{aligned}
\end{equation*}
Using those notations, any feasible point to~\eqref{eq:relaxed} can be transformed to a feasible point to the following~\eqref{eq:SDP_PEP}, using the Gram matrix representation. Hence, the optimal value to the following problem is an upper bound on that of~\eqref{eq:relaxed}. Note that an argument along the lines of \cite[Theorem~5]{taylor2017smooth} can be used to establish that the optimal value to this problem is actually equal to that of~\eqref{eq:relaxed}, however, as we only exploit the upper bound below we will omit the proof for this property.
\begin{align}\max_{\substack{G\in\mathbb{S}^{N+1}\\F\in\mathbb{R}^{N}\\d\in\mathbb{N}}} & \bxx_N^\top G \bxx_N&&\notag \\
\text{s.t. }& G\succeq 0,\label{eq:SDP_PEP}\tag{SDP-R}\\
& \bxx_0^\top G \bxx_0=1,&&\notag\\
& 0\geq  (\bff_{i+1}-\bff_i)^\top F+ \bgg_{i+1}^\top G(\bxx_i-\bxx_{i+1}) +\tfrac1{2(L-\mu)}(\bgg_i-\bgg_{i+1})^\top G(\bgg_i-\bgg_{i+1}) &&\text{for }i=0,\hdots,N-2,\notag \\
& 0\geq  \bff_i^\top F-\bgg_{i}^\top G\bxx_{i} +\tfrac1{2(L-\mu)}\bgg_i^\top G\bgg_i &&\text{for }i=0,\hdots,N-1,\notag \\
& 0\geq  -\bff_{N-1}^\top F+\tfrac1{2(L-\mu)}\bgg_{N-1}^\top G\bgg_{N-1},&&\notag \\
& \mathrm{rank}(G)\leq d.&&\notag
\end{align}
After getting rid of the variable $d$ and the rank constraint (which is void due to maximization over $d$){, this problem is a {linear} SDP, parametrized by $L>\mu\geq0$, and $\{\alpha_{i,j}\}$.}

For transforming the minimax problem to a bilinear minimization problem, the next key step in our procedure is to express the Lagrangian dual of~\eqref{eq:SDP_PEP}, { substituting the inner maximization problem by a minimization, hence replacing the minimax by a minimization problem. Note that we do not assume strong duality, as weak duality suffices for obtaining an upper bound on the original problem}. That is, we perform the following primal-dual associations
\begin{equation*}
\begin{aligned}
&\|w_0-w_\star\|^2=1 && &&:\tau,\\
& f_i\geq f_{i+1}+\langle g_{i+1};w_i-w_{i+1}\rangle +\tfrac1{2(L-\mu)}\|g_i-g_{i+1}\|^2 &&\text{for }i=0,\hdots,N-2&&:\lambda_{i,i+1},\\
& f_\star\geq f_i+\langle g_i,w_\star-w_{i}\rangle +\tfrac1{2(L-\mu)}\|g_{i}\|^2 &&\text{for }i=0,\hdots,N-1&&:\lambda_{\star,i},\\
& f_{N-1}\geq f_\star+\tfrac1{2(L-\mu)}\|g_{N-1}\|^2 && &&:\lambda_{N-1,\star},
\end{aligned}    
\end{equation*}
and arrive to the following dual formulation of~\eqref{eq:SDP_PEP}, whose optimal value is denoted by $\mathrm{UB}_{\mu,L}(M)$
\begin{align}
\mathrm{UB}_{\mu,L}(\{\alpha_{i,j}\}):=\min_{\tau,\lambda_{i,j}\geq0}& \tau,&&\notag\\\
\text{s.t.}&\,S(\tau,\{\lambda_{i,j}\},\{\alpha_{i,j}\})\succeq 0,&&\label{eq:dual_relax_sdp}\tag{dual-SDP-R}\\
&\sum_{i=0}^{N-2} \lambda_{i,i+1}(\bff_{i+1}-\bff_i)+\sum_{i=0}^{N-1}\lambda_{\star,i}\bff_i-\lambda_{N-1,\star}\bff_{N-1}=0\notag,
\end{align}
with (note the dependence on $\{\alpha_{i,j}\}$ via $\bxx_i$'s)
\begin{equation*}
\begin{aligned}
S(\tau,\{\lambda_{i,j}\}{,\{\alpha_{i,j}\}})=&\tau\, \bxx_0\bxx_0^\top-\bxx_{N}\bxx_{N}^\top+\frac{\lambda_{N-1,\star}}{2(L-\mu)}\bgg_{N-1}\bgg_{N-1}^\top+\sum_{i=0}^{N-1} \frac{\lambda_{\star,i}}{2}\left(-\bgg_i \bxx_i^\top -\bxx_i\bgg_i^\top+\tfrac{1}{L-\mu}\bgg_i\bgg_i^\top\right)\\
&+\sum_{i=0}^{N-2}\frac{\lambda_{i,i+1}}{2}\left(\bgg_{i+1}(\bxx_i-\bxx_{i+1})^\top+(\bxx_i-\bxx_{i+1})\bgg_{i+1}^\top+\tfrac{1}{L-\mu}(\bgg_i-\bgg_{i+1})(\bgg_i-\bgg_{i+1})^\top\right).
\end{aligned}
\end{equation*}
Note that in the case $N>1$, the equality constraint in~\eqref{eq:dual_relax_sdp} can be written as
\begin{equation}\label{eq:dual_lin_cons}
\begin{aligned}
& \lambda_{\star,0}-\lambda_{0,1}=0,&&\\
& \lambda_{i-1,i}+\lambda_{\star,i}-\lambda_{i,i+1}=0 \quad &&\text{for }i=1,\hdots,N-2,\\
& \lambda_{N-2,N-1}+\lambda_{\star,N-1}-\lambda_{N-1,\star}=0. &&
\end{aligned}
\end{equation}
When $N=1$, it reduces to $\lambda_{\star,0}-\lambda_{0,\star}=0$. The following lemma recaps the current situation.

\begin{lemma}\label{lem:ub_sdp} Let $N\in\mathbb{N}$, $0\leq\mu<L<\infty$, and $M\in\mathcal{M}_N$ be a black-box first-order method~\eqref{eq:method_reparam} performing $N$ gradient evaluations and described by a set of coefficients $\{\alpha_{i,j}\}_{i,j}$. For any $d\in\mathbb{N}$, $w_0\in\mathbb{R}^d$, $f\in\mathcal{F}_{\mu,L}(\mathbb{R}^d)$, $w_\star\in\mathrm{argmin}_w f(w)$, and $w_N=M(w_0,f)$, it holds that
\[\|w_N-w_\star\|^2\leq \mathrm{UB}_{\mu,L}(\{\alpha_{i,j}\})\, \|w_0-w_\star\|^2.\]
\end{lemma}
\begin{proof} The result follows Lemma~\ref{lem:pep}. More precisely, any feasible point to~\eqref{eq:relaxed} can be translated to a feasible point to~\eqref{eq:SDP_PEP} using the Gram matrix representation~\eqref{eq:gram_representation}, hence $\mathrm{val}\text{\eqref{eq:relaxed}}\leq \mathrm{val}\text{\eqref{eq:SDP_PEP}}$. Furthermore, weak duality implies $\mathrm{val}\text{\eqref{eq:SDP_PEP}}\leq\mathrm{val}\text{\eqref{eq:dual_relax_sdp}}= \mathrm{UB}_{\mu,L}(\{\alpha_{i,j}\})$. Therefore
\[ \mathrm{val}\text{\eqref{eq:relaxed}}\leq \mathrm{val}\text{\eqref{eq:SDP_PEP}}\leq \mathrm{UB}_{\mu,L}(\{\alpha_{i,j}\}),\]
and it follows that
\[\|w_N-w_\star\|^2\leq \mathrm{val}\text{\eqref{eq:relaxed}}\|w_0-w_\star\|^2\leq \mathrm{UB}_{\mu,L}(\{\alpha_{i,j}\})\, \|w_0-w_\star\|^2,\]
where the first inequality is due to Lemma~\ref{lem:pep}.\qed
\end{proof}

The last remaining difficulty is that $S({\cdot})$ appearing in~\eqref{eq:dual_relax_sdp} is bilinear in terms of the algorithmic parameters $\{\alpha_{i,j}\}$ (the vectors $\bxx_i$ depend linearly on those parameters) and the dual variables $\{\lambda_{i,j}\}$. Therefore, it might be unclear how to efficiently solve
\begin{equation*}
\begin{aligned}
\min_{\{\alpha_{i,j}\}} \mathrm{UB}_{\mu,L}(\{\alpha_{i,j}\})\equiv \min_{\{\alpha_{i,j}\}}\,\min_{\tau,\{\lambda_{i,j}\}\geq0}& \tau,&&\\
\text{s.t.}&\,S(\tau,\{\lambda_{i,j}\},\{\alpha_{i,j}\})\succeq 0,&&\\
&\sum_{i=0}^{N-2} \lambda_{i,i+1}(\bff_{i+1}-\bff_i)+\sum_{i=0}^{N-1}\lambda_{\star,i}\bff_i-\lambda_{N-1,\star}\bff_{N-1}=0,
\end{aligned}
\end{equation*}
as problems involving such bilinear matrix inequalities are NP-hard in general. In the next section, we employ a \emph{linearization} trick which allows tackling this specific problem.

\subsection{An approximate minimax and its semidefinite representation}\label{s:approx_minimax}
In this section, we proceed with the last stage of our construction, by showing how to solve
\begin{equation}\label{eq:min_upperbound}
\min_{\{\alpha_{i,j}\}} \mathrm{UB}_{\mu,L}(\{\alpha_{i,j}\}),    
\end{equation}
which is a minimization problem jointly on $\tau$, $\alpha_{i,j}$'s and $\lambda_{i,j}$'s.  As it is, the problem features a \emph{bilinear matrix inequality}. A few algebraic manipulations on the matrix $S$ allows rewriting the bilinear matrix inequality in terms of a matrix $S'$, in a slightly more explicit and convenient way (those manipulations are provided in Appendix~\ref{s:algebra}, where $S'$ is defined). This structure reveals that the following change of variables allows linearizing the bilinear matrix inequality {(for all $0\leq j<i=1,\ldots,N$)}
\begin{equation}\label{eq:design_SDP_dist}
\begin{aligned}
\beta_{i,j}=\left\{\begin{array}{ll}
\lambda_{i,i+1}\alpha_{i,j}   \quad  & \text{if } 0< i <N-1{,} \\
 \lambda_{N-1,\star}\alpha_{N-1,j}    & \text{if } i=N-1{,}\\
 \alpha_{N,j} &\text{if } i=N.
\end{array}\right.
\end{aligned}    
\end{equation}
As provided in Lemma~\ref{lem:final_construction} below, this change of variables is invertible for the problem under consideration. In other words, for any $N>1$ and $0\leq\mu<L$, one can solve~\eqref{eq:min_upperbound} via its reformulation (intermediate computations, involving a matrix $S'$ are provided in Appendix~\ref{s:algebra}; the important thing to see about those formulation is how variables $\{\alpha_{i,j}\}$ and $\{\lambda_{i,j}\}$ interact with each others), as
\begin{equation}\label{eq:minibound}\tag{Minimax-R}
\begin{aligned}
\min_{\substack{\tau,\{\lambda_{i,j}\}\geq0\\\{\beta_{i,j}\}}}& \tau&&\\
\text{s.t.}&\,\begin{pmatrix}S''(\tau,\{\lambda_{i,j}\},\{\beta_{i,j}\}) & \bxx_N \\ \bxx_N^\top & 1\end{pmatrix}\succeq 0,&&\\
&\,\sum_{i=0}^{N-2} \lambda_{i,i+1}(\bff_{i+1}-\bff_i)+\sum_{i=0}^{N-1}\lambda_{\star,i}\bff_i-\lambda_{N-1,\star}\bff_{N-1}=0,
\end{aligned}    
\end{equation}
which is a standard linear semidefinite program, with the following definitions (note the linear dependencies in all parameters $\tau,\{\lambda_{i,j}\},\{\beta_{i,j}\}$)
\begin{equation*}
\begin{aligned}
S''(\tau,\{\lambda_{i,j}\},&\{\beta_{i,j}\})\\=\,&\tfrac{1}{2(L-\mu)}\left(\lambda_{N-1,\star}\,\bgg_{N-1}\bgg_{N-1}^\top+\sum_{i=0}^{N-1}\lambda_{\star,i}\bgg_i\bgg_i^\top+\sum_{i=0}^{N-2}\lambda_{i,i+1}(\bgg_i-\bgg_{i+1})(\bgg_i-\bgg_{i+1})^\top\right)\\
&-\tfrac12\lambda_{\star,0}(\bgg_0\bxx_0^\top+\bxx_0\bgg_0^\top)-\sum_{i=1}^{N-2}\left(\lambda_{i,i+1}-\tfrac{\mu}{L}\sum_{j=0}^{i-1}\beta_{i,j}\right) \,\tfrac12 (\bgg_i\bxx_0^\top+\bxx_0\bgg_i^\top)\\&+\sum_{i=1}^{N-2}\sum_{j=0}^{i-1}\tfrac{\beta_{i,j}}{L}\,\tfrac12(\bgg_i\bgg_j^\top+\bgg_j\bgg_i^\top)-\left(\lambda_{N-1,\star}-\tfrac{\mu}{L}\sum_{j=0}^{N-2}\beta_{N-1,j}\right)\,\tfrac12(\bgg_{N-1}\bxx_0^\top+\bxx_0\bgg_{N-1}^\top)\\&+\sum_{j=0}^{N-2}\tfrac{\beta_{N-1,j}}{L}\,\tfrac12(\bgg_{N-1}\bgg_j^\top+\bgg_j\bgg_{N-1}^\top)+\sum_{i=0}^{N-2}\left(\lambda_{i,i+1}-\tfrac{\mu}{L}\sum_{j=0}^{i-1}\beta_{i,j}\right)\,\tfrac12(\bgg_{i+1}\bxx_0^\top+\bxx_0\bgg_{i+1}^\top)\\
&-\sum_{i=0}^{N-2}\sum_{j=0}^{i-1}\tfrac{\beta_{i,j}}{L}\,\tfrac12(\bgg_{i+1}\bgg_j^\top+\bgg_j\bgg_{i+1}^\top)+\tau \bxx_0\bxx_0^\top,
\end{aligned}
\end{equation*}
and $\bxx_N=\,\bxx_0\left(1-\tfrac{\mu}{L}\sum_{i=0}^{N-1}\alpha_{N,i}\right)-\sum_{i=0}^{N-1}\tfrac{\alpha_{N,i}}{L}\bgg_i$. From the solution to~\eqref{eq:minibound}, one can recover a fixed-step first-order method whose worst-case performance satisfies
\[ \|x_N-x_\star\|^2\leq\mathrm{val}\text{\eqref{eq:minibound}}\, \|x_0-x_\star\|^2,\]
as formalized by the next lemma.

\begin{lemma}\label{lem:final_construction}Let $N\in\mathbb{N}$ with $N>1$, and $0\leq \mu<L<\infty$. Furthermore, let $(\tau,\{\beta_{i,j}\},\{\lambda_{i,j}\})$ be a solution to~\eqref{eq:minibound}. The following statements hold.
\begin{itemize}
    \item[(i)] If $\lambda_{{k,k+1}}= 0$ {for some $k\in\{1,\ldots,N-2\}$}, then $\beta_{{k,j}}=0$ {for all $j=0,\ldots,k-1$}.
    \item[(ii)] If $\lambda_{N-1,\star}= 0$, then $\beta_{N-1,j}=0$ {for all $j=0,\ldots,N-2$}.
    \item[(iii)] Let $\{\alpha_{i,j}\}$ be defined as
    \begin{equation}\label{eq:design_SDP_dist_inv}
\begin{aligned}
\alpha_{i,j}=\left\{\begin{array}{ll}
0 \quad &\text{if } \beta_{i,j}=0,\\
{\beta_{i,j}}/{\lambda_{i,i+1}}   \quad  & \text{if } 0{<} i <N-1 {\text{ and } 0\leq j<i}, \\
{\beta_{N-1,j}}/{ \lambda_{N-1,\star}}   & \text{if } i=N-1  {\text{ and } 0\leq j<i},\\
\beta_{N,j} &\text{if } i=N  {\text{ and } 0\leq j<i}.
\end{array}\right.
\end{aligned}    
\end{equation}
The output of the corresponding method of the form~\eqref{eq:method_reparam} satisfies
\[ \|w_N-w_\star\|^2\leq \tau \|w_0-w_\star\|^2\]
    for any $d\in\mathbb{N}$, $w_0\in\mathbb{R}^d$, and $f\in\mathcal{F}_{\mu,L}(\mathbb{R}^d)$.
\end{itemize}
\end{lemma}
\begin{proof}
(i) assume $\lambda_{k,k+1}=0$ for some $k>0$; it follows from $\lambda_{i,j}\geq 0$ and~\eqref{eq:dual_lin_cons} that $\lambda_{\star,i}=0$ {(for all $0\leq i\leq k$)} and $\lambda_{i-1,i}=0$ {(for all $1\leq i\leq k$)}. From the expression of $S''$, it means that there are no diagonal entries corresponding to the entries $e_2 e_2^\top$, ..., $e_{k+2}e_{k+2}^\top$ (corresponding to $\bgg_0\bgg_0^\top$, ..., $\bgg_k\bgg_k^\top$). Therefore, the constraint $S''\succeq 0$ imposes the corresponding off-diagonal elements to be equal to zero as well (i.e., all entries corresponding to $\bxx_0\bgg_0^\top,\hdots,\bxx_0\bgg_k^\top$ and $\bgg_i\bgg_j^\top$ for $j=0,\hdots,k$ and $i=0,\hdots,N-1$). It follows {from a short recurrence argument} that {$\beta_{i,0},\hdots,\beta_{i,i-1}=0$ for all $1\leq i\leq k$ and hence in particular that $\beta_{k,0},\hdots,\beta_{k,k-1}=0$}.

(ii) Using a similar argument: it follows from $\lambda_{N-1,\star}=0$ that $\lambda_{N-2,N-1}=\lambda_{\star,N-1}=0$. There is therefore no diagonal element corresponding to the entry $\bgg_{N-1}\bgg_{N-1}^\top$ in $S''$, and the corresponding off-diagonal elements should be zero as well due to the constraint $S''\succeq 0$. {Hence, together with (i), the last argument allows to conclude that $\beta_{N-1,0},\hdots,\beta_{N-1,N-2}=0$}.

(iii) Using (i) and (ii), and for $\{\alpha_{i,j}\}$ {(which is well defined due to (i) and (ii))}, the couple $(\tau,\{\lambda_{i,j}\})$ is feasible for~\eqref{eq:dual_relax_sdp} by construction, following the reformulation steps of $S$ in Appendix~\ref{s:algebra}. It follows that $\mathrm{UB}_{\mu,L}(\{\alpha_{i,j}\})\leq \tau$ and Lemma~\ref{lem:ub_sdp} allows reaching the desired claim.
\qed
\end{proof}
It is relatively straightforward to establish that \shortMethodName{} is a solution to~\eqref{eq:design_SDP_dist}, given that~(i) \shortMethodName{} achieves the lower complexity bound (see Section~\ref{s:LB}), that~(ii) \shortMethodName{} is a fixed-step first-order method, and that~(iii) all the inequalities involved in the proof of Theorem~\ref{thm:potential} and Theorem~\ref{corr:finalbound} are used in the relaxation procedure~\eqref{eq:relaxed}.  We conclude this section by the corresponding formal statement. 
\begin{theorem}\label{thm:item_is_a_sol} Let $N\in\mathbb{N}$, and $0\leq \mu<L<\infty$. Algorithm~\eqref{eq:OGM_sc} is a solution to~\eqref{eq:minibound}.
\end{theorem}
\begin{proof}
We exhibit a solution to~\eqref{eq:minibound} and show that it corresponds to \shortMethodName{} in Appendix~\ref{a:minimax_solution}.\qed
\end{proof}

Numerical examples of the design procedure are provided in Appendix~\ref{a:numerics}, including optimized methods for different objectives, like function values, for which we provide the slightly adapted design strategy in Appendix~\ref{a:func_values}. Code for reproducing the results are provided in Section~\ref{s:ccl}.

{\begin{remark}[From numerical values of the step sizes to an analytical algorithm] Before concluding, let us informally mention a few details on how we 
obtained the analytical formulation for Algorithm~\ref{alg:item} (ITEM) from the numerical values of the step sizes. 


The main observations that helped are the following: (i) the step size policy appeared not to depend on the horizon $N$, (ii) the numerics suggested that the optimal values of the step sizes could be ``factored'' in an efficient form not requiring to store all previous gradients (see, e.g., discussions in~\citep[Section 5]{drori2020efficient}), and (iii) only a few inequalities were active at the optimal point. Those three observations suggested the existence of a potential/Lyapunov-based proof structure for the optimized method. Perhaps luckily, this allowed us to engineer a method matching the coefficients and worst-case bounds obtained from the numerical step size optimization procedure, as well as the results from the analytical lower complexity bound (see~\citep[Corollary 4]{drori2021exact}).
\end{remark}}
\section{Conclusion}\label{s:ccl}
In this work, we provided the \emph{\fullMethodName{ }} (\shortMethodName), a first-order method whose worst-case guarantee exactly matches the lower bound for minimizing smooth strongly convex functions. Furthermore, we showed how to develop such methods {constructively}, through \emph{performance estimation problems} and semidefinite programming.

We believe that obtaining accelerated first-order methods as solutions to minimax problems certainly brings perspectives and a systematic approach to accelerated methods in first-order convex optimization, similar to the design procedure for obtaining Chebyshev methods for quadratic minimization (see e.g., the survey of~\citep{dAspremontAccelerated}). In addition, we think that the conceptual simplicity of the shapes and proofs of such optimized methods render them attractive as textbook examples for illustrating the acceleration phenomenon. In particular, it appeared as very surprising to us that both sequences $\{y_k\}$ and $\{z_k\}$ now have relatively clear interpretations: $z_k$'s are optimal for optimizing the distance to an optimal solution, whereas $y_k$'s are essentially optimal for optimizing function values (see~\citep{kim2015optimized}). 

Those methods might as well serve as an inspiration for further developments on this topic, for designing accelerated methods in other settings, and for alternate performance criterion, as showcased numerically in Appendix~\ref{a:numerics}.

Finally, let us mention that extending methods such as the Optimized Gradient Method, the Triple Momentum Method, and the \fullMethodName{ }to more general situations, possibly involving constraints, for instance, seems less straightforward compared to other acceleration schemes. We leave this question for future works.

\paragraph{Software} Source code for helping in reproducing the slightly algebraic passage in~\S\ref{s:OGM_sc} can be found in
\begin{center}
\url{https://github.com/AdrienTaylor/Optimal-Gradient-Method}    
\end{center}
together with implementations in the Performance Estimation Toolbox~\citep{taylor2017performance} for validating the potential from Lemma~\ref{thm:potential}, bounds from Theorem~\ref{corr:finalbound}, and the constructive procedure of Lemma~\ref{lem:final_construction}.

\paragraph{Acknowledgments} {The authors would like to thank \href{https://shuvomoy.github.io/site/}{Shuvomoy Das Gupta} for pointing out a few typos and for nice suggestions for improvements. We are also very grateful to two anonymous referees and an associate editor of Mathematical Programming for providing very constructive feedbacks on an early version of the manuscript.}

\bibliographystyle{plainnat}
\bibliography{bib}{}   

\appendix
\section{Alternate parametrization for first-order methods}\label{s:reparam}
In this section, we show that any method~\eqref{eq:method} can be reparametrized as~\eqref{eq:method_reparam} (and vice-versa), using the identity~\eqref{eq:equiv_param}. First note that the equivalence is clear for $k=1$, as
\begin{equation*}
\begin{aligned}
w_1-x_\star&=w_0-x_\star-\tfrac{h_{1,0}}{L}(\nabla\tilde f(w_0)+\mu (w_0-x_\star))\\
&=(w_0-x_\star)(1-\tfrac{\mu}{L}h_{1,0})-\tfrac{h_{1,0}}{L}\nabla\tilde f(w_0),
\end{aligned}
\end{equation*}
and hence the equivalence holds for $k=1$. Now, assuming the equivalence holds at iteration $k$, let us check that it holds at iteration $k+1$, that is assume $w_i-x_\star=(w_0-x_\star)\left(1-\tfrac{\mu}{L}\sum_{j=0}^{i-1}\alpha_{i,j}\right)-\tfrac{1}{L}\sum_{j=0}^{i-1}\alpha_{i,j}\nabla\tilde f(w_j)$ for $0\leq i\leq k$ and compute
\begin{equation*}
\begin{aligned}
w_{k+1}-x_\star=&w_k-x_\star-\tfrac{1}{L}\sum_{i=0}^k h_{k+1,i}(\nabla \tilde f(w_i)+\mu(w_i-x_\star))\\
=&(w_0-x_\star)\left(1-\tfrac{\mu}{L}\sum_{i=0}^{k-1}\alpha_{k,i}-\tfrac{\mu}{L}\sum_{i=0}^k h_{k+1,i}+\tfrac{\mu^2}{L^2}\sum_{i=0}^k \sum_{j=0}^{k-1} h_{k+1,i}\alpha_{i,j}\right)\\
&-\tfrac{1}{L}\sum_{i=0}^{k-1}\alpha_{k,i}\nabla\tilde f(w_i)-\tfrac{1}{L}\sum_{i=0}^k h_{k+1,i}\nabla \tilde f(w_i)+\tfrac{\mu}{L^2}\sum_{i=0}^k \sum_{j=0}^{k-1}h_{k+1,i}\alpha_{i,j}\nabla \tilde f(w_j)
\end{aligned}
\end{equation*}
by reverting the ordering of the double sums, renaming the indices, and reordering, we get
\begin{equation*}
\begin{aligned}
w_{k+1}-x_\star=&(w_0-x_\star)\left(1-\tfrac{\mu}{L}\sum_{i=0}^{k-1}\alpha_{k,i}-\tfrac{\mu}{L}\sum_{i=0}^k h_{k+1,i}+\tfrac{\mu^2}{L^2}\sum_{j=0}^{k-1} \sum_{i=j+1}^{k} h_{k+1,i}\alpha_{i,j}\right)\\
&-\tfrac{1}{L}\sum_{i=0}^{k-1}\alpha_{k,i}\nabla\tilde f(w_i)-\tfrac{1}{L}\sum_{i=0}^k h_{k+1,i}\nabla \tilde f(w_i)+\tfrac{\mu}{L^2}\sum_{j=0}^{k-1} \sum_{i=j+1}^{k}h_{k+1,i}\alpha_{i,j}\nabla \tilde f(w_j)\\
=&(w_0-x_\star)\left(1-\tfrac{\mu}{L}\sum_{i=0}^{k-1}\alpha_{k,i}-\tfrac{\mu}{L}\sum_{i=0}^k h_{k+1,i}+\tfrac{\mu^2}{L^2}\sum_{i=0}^{k-1} \sum_{j=i+1}^{k} h_{k+1,j}\alpha_{j,i}\right)\\
&-\tfrac{1}{L}\sum_{i=0}^{k-1}\alpha_{k,i}\nabla\tilde f(w_i)-\tfrac{1}{L}\sum_{i=0}^k h_{k+1,i}\nabla \tilde f(w_i)+\tfrac{\mu}{L^2}\sum_{i=0}^{k-1} \sum_{j=i+1}^{k}h_{k+1,j}\alpha_{j,i}\nabla \tilde f(w_i)\\
=&(w_0-x_\star)\left(1-\tfrac{\mu}{L}h_{k+1,k}-\tfrac{\mu}{L}\sum_{i=0}^{k-1}\left(\alpha_{k,i}+ h_{k+1,i}-\tfrac{\mu}{L}\sum_{j=i+1}^{k} h_{k+1,j}\alpha_{j,i}\right)\right)\\
&-\tfrac1L h_{k+1,k}\nabla \tilde f(w_k) -\tfrac{1}{L}\sum_{i=0}^{k-1}\left(\alpha_{k,i}+h_{k+1,i}-\tfrac{\mu}{L}\sum_{j=i+1}^{k}h_{k+1,j}\alpha_{j,i}\right)\nabla \tilde f(w_i).
\end{aligned}
\end{equation*}
From this last reformulation, the choice~\eqref{eq:equiv_param}, that is
\begin{equation*}
\begin{aligned}
\alpha_{k+1,i}=\left\{\begin{array}{ll}
h_{k+1,k}     & \text{if } i=k \\
h_{k+1,i}+\alpha_{k,i}-\tfrac{\mu}{L}\sum_{j=i+1}^k h_{k+1,j}\alpha_{j,i}\quad    & \text{if } 0\leq i <k,
\end{array}\right.
\end{aligned}
\end{equation*}
allows enforcing the coefficients of all independent terms $(w_0-x_\star)$, $\nabla \tilde f(w_0),\hdots,\nabla \tilde f(w_k)$ to be equal in both~\eqref{eq:method} and~\eqref{eq:method_reparam}, reaching the desired statement. In addition, note that this change of variable is {reversible}.

\section{Algebraic manipulations of~\eqref{eq:dual_relax_sdp}}\label{s:algebra}
In this section, we reformulate~\eqref{eq:dual_relax_sdp} for enabling us optimizing both on $\alpha_{i,j}$'s and $\lambda_{i,j}$'s simultaneously. For doing that, let us start by conveniently noting that
\[ S(\tau,\{\lambda_{i,j},\{\alpha_{i,j}\}\})\succeq 0 \Leftrightarrow \begin{pmatrix}S'(\tau,\{\lambda_{i,j}\},\{\alpha_{i,j}\}) & \bxx_N \\ \bxx_N^\top & 1\end{pmatrix}\succeq 0, \]
with $S'(\tau,\{\lambda_{i,j}\},\{\alpha_{i,j}\})=S(\tau,\{\lambda_{i,j}\},\{\alpha_{i,j}\})+\bxx_N\bxx_N^\top$, using a standard Schur complement (see, e.g.,~\citep{van1983matrix}). The motivation underlying this reformulation is that this \emph{lifted} linear matrix inequality depends linearly on $\{\alpha_{N,i}\}_i$'s. Indeed, the coefficients of the last iteration only appear through the term $\bxx_N$, which is not present in $S'$ (details below).

We only consider the case $N>1$ below. In the case $N=1$,{~\eqref{eq:blackbox_optim}} can be solved without the following simplifications. Let us develop the expression of $S'({\cdot})$ as follows
\begin{equation*}
\begin{aligned}
S'(\tau,\{\lambda_{i,j}\},\{\alpha_{i,j}\})=\,&\tau \bxx_0\bxx_0^\top+\tfrac{1}{2(L-\mu)}\left(\lambda_{N-1,\star}\,\bgg_{N-1}\bgg_{N-1}^\top+\sum_{i=0}^{N-1}\lambda_{\star,i}\bgg_i\bgg_i^\top+\sum_{i=0}^{N-2}\lambda_{i,i+1}(\bgg_i-\bgg_{i+1})(\bgg_i-\bgg_{i+1})^\top\right)\\
&-\tfrac12\lambda_{\star,0}(\bgg_0\bxx_0^\top+\bxx_0\bgg_0^\top)-\tfrac12\sum_{i=1}^{N-1}(\lambda_{i-1,i}+\lambda_{\star,i})(\bgg_i\bxx_i^\top+\bxx_i\bgg_i^\top)\\&+\tfrac12\sum_{i=0}^{N-2}\lambda_{i,i+1}(\bgg_{i+1}\bxx_i^\top+\bxx_i\bgg_{i+1}^\top)\\
=\,&\tau \bxx_0\bxx_0^\top+\tfrac{1}{2(L-\mu)}\left(\lambda_{N-1,\star}\,\bgg_{N-1}\bgg_{N-1}^\top+\sum_{i=0}^{N-1}\lambda_{\star,i}\bgg_i\bgg_i^\top+\sum_{i=0}^{N-2}\lambda_{i,i+1}(\bgg_i-\bgg_{i+1})(\bgg_i-\bgg_{i+1})^\top\right)\\
&-\tfrac12\lambda_{\star,0}(\bgg_0\bxx_0^\top+\bxx_0\bgg_0^\top)-\tfrac12\sum_{i=1}^{N-2}\lambda_{i,i+1}(\bgg_i\bxx_i^\top+\bxx_i\bgg_i^\top)\\&-\tfrac12\lambda_{N-1,\star}(\bgg_{N-1}\bxx_{N-1}^\top+\bxx_{N-1}\bgg_{N-1}^\top)+\tfrac12\sum_{i=0}^{N-2}\lambda_{i,i+1}(\bgg_{i+1}\bxx_i^\top+\bxx_i\bgg_{i+1}^\top),
\end{aligned}
\end{equation*}
where we used $\lambda_{i-1,i}+\lambda_{\star,i}=\lambda_{i,i+1}$ (for $i=1,\hdots,N-2$), and $\lambda_{N-2,N-1}+\lambda_{\star,N-1}=\lambda_{N-1,\star}$ for obtaining the second equality. Substituting the expressions for $\bxx_i$'s, we arrive to
\begin{equation*}
\begin{aligned}
S'(\tau,\{\lambda_{i,j}\},\{\alpha_{i,j}\})=\,&\tau \bxx_0\bxx_0^\top+\tfrac{1}{2(L-\mu)}\left(\lambda_{N-1,\star}\,\bgg_{N-1}\bgg_{N-1}^\top+\sum_{i=0}^{N-1}\lambda_{\star,i}\bgg_i\bgg_i^\top+\sum_{i=0}^{N-2}\lambda_{i,i+1}(\bgg_i-\bgg_{i+1})(\bgg_i-\bgg_{i+1})^\top\right)\\
&-\tfrac12\lambda_{\star,0}(\bgg_0\bxx_0^\top+\bxx_0\bgg_0^\top)-\sum_{i=1}^{N-2}\lambda_{i,i+1}\left(1-\tfrac{\mu}{L}\sum_{j=0}^{i-1}\alpha_{i,j}\right) \,\tfrac12 (\bgg_i\bxx_0^\top+\bxx_0\bgg_i^\top)\\&+\sum_{i=1}^{N-2}\lambda_{i,i+1}\sum_{j=0}^{i-1}\tfrac{\alpha_{i,j}}{L}\,\tfrac12(\bgg_i\bgg_j^\top+\bgg_j\bgg_i^\top)-\lambda_{N-1,\star}\left(1-\tfrac{\mu}{L}\sum_{j=0}^{N-2}\alpha_{N-1,j}\right)\,\tfrac12(\bgg_{N-1}\bxx_0^\top+\bxx_0\bgg_{N-1}^\top)\\&+\lambda_{N-1,\star}\sum_{j=0}^{N-2}\tfrac{\alpha_{N-1,j}}{L}\,\tfrac12(\bgg_{N-1}\bgg_j^\top+\bgg_j\bgg_{N-1}^\top)+\sum_{i=0}^{N-2}\lambda_{i,i+1}\left(1-\tfrac{\mu}{L}\sum_{j=0}^{i-1}\alpha_{i,j}\right)\,\tfrac12(\bgg_{i+1}\bxx_0^\top+\bxx_0\bgg_{i+1}^\top)\\
&-\sum_{i=0}^{N-2}\lambda_{i,i+1}\sum_{j=0}^{i-1}\tfrac{\alpha_{i,j}}{L}\,\tfrac12(\bgg_{i+1}\bgg_j^\top+\bgg_j\bgg_{i+1}^\top),
\end{aligned}
\end{equation*}
where we simply expressed each $\bxx_k$ in two terms, one with the contribution of $\bxx_0$, and the other with the contributions of $\bgg_i$'s. Although not pretty, one can observe that $S'$ is still bilinear in $\{\alpha_{i,j}\}$ and $\{\lambda_{i,j}\}$. This expression can be largely simplified, but this form suffices for the purposes in this work. 
\section{The \fullMethodName{} is a solution to~\eqref{eq:minibound}}\label{a:minimax_solution}

\paragraph{Proof of Theorem~\ref{thm:item_is_a_sol}} For readability purposes, we establish the claim without explicitly computing the optimal values of the variables $\{\alpha_{i,j}\}$ and $\{\beta_{i,j}\}$ for \shortMethodName{}. For avoiding this step, let us note that \shortMethodName{} is clearly a fixed-step first-order method following Definition \ref{def:FS}. Therefore, following Lemma~\ref{lem:eq_param}, the method can also be written in the alternate parametrization~\eqref{eq:method_reparam}, using the association $w_i\leftarrow y_i$ ($i=0,1,\hdots,N-1$) and $w_N\leftarrow z_N$, where $y_k$ and $z_k$ are the sequences defined by~\eqref{eq:OGM_sc}.
Let $\{\alpha_{i,j}^\star\}$ denote the steps sizes corresponding to ITEM written in the form~\eqref{eq:method_reparam}. We proceed to show that by choosing
\begin{equation}\label{eq:def_item_lambdas}
\begin{aligned}
{\lambda_{\star,i}^\star}&=\frac{1-\cond}{L}\frac{A_{i+1}-A_i}{1+\cond A_N}\\
{\lambda_{i,i+1}^\star}&=\frac{1-\cond}{L}\frac{A_{i+1}}{1+\cond A_N}\\
{\lambda_{N-1,\star}^\star}&=\frac{1-\cond}{L}\frac{A_N}{1+\cond A_N}\\
{\tau^\star} &= \frac{1}{1+\cond A_N}
\end{aligned}
\end{equation}
and setting $\{\beta_{i,j}^{\star}\}$ in accordance to \eqref{eq:design_SDP_dist}, we reach a feasible solution to \eqref{eq:minibound}.
Note that optimality of the solution follows from the value of $\tau^\star$, which matches the lower complexity bound discussed in Section~\ref{s:LB}.

For establishing dual feasibility, we relate~\eqref{eq:minibound} to the Lagrangian of~\eqref{eq:relaxed}. That is, denoting
\[ K=S''(\tau^\star,\{\lambda_{i,j}^\star\},\{\beta_{i,j}^{\star}\})-\bxx_N\bxx_N^\top=S(\tau^\star,\{\lambda_{i,j}^\star\},\{\alpha_{i,j}^\star\}),\]
we have for all $(F,G)$ as in~\eqref{eq:gram_representation}, by construction,
\begin{equation*}
\begin{aligned}
&F\left(\sum_{i=0}^{N-2} \lambda_{i,i+1}^\star(\bff_{i+1}-\bff_i)+\sum_{i=0}^{N-1}\lambda_{\star,i}^\star\bff_i-\lambda_{N-1,\star}^\star\bff_{N-1}\right)+\mathrm{Tr}(KG)\\
&\,=\tau^\star \|w_0-w_\star\|^2-\|w_N-w_\star\|^2\\
&\quad +\sum_{i=0}^{N-2}\lambda_{i,i+1}^\star\left[ \tilde f(w_{i+1})- \tilde f(w_i)+\langle \nabla \tilde f(w_{i+1});w_i-w_{i+1}\rangle +\tfrac1{2(L-\mu)}\|\nabla \tilde f (w_i)-\nabla \tilde f(w_{i+1})\|^2\right]\\&\quad+\sum_{i=0}^{N-1}{\lambda_{\star,i}^\star}\left[\tilde f(w_i)-\tilde f_\star+\langle \nabla \tilde f(w_i),w_\star-w_{i}\rangle +\tfrac1{2(L-\mu)}\|\nabla \tilde f(w_{i})\|^2\right]\\&\quad+{\lambda_{N-1,\star}^\star}\left[\tilde f_\star-\tilde f(w_{N-1})+\tfrac1{2(L-\mu)}\|\nabla \tilde f(w_{N-1})\|^2\right].
\end{aligned}
\end{equation*}
Using the association $w_i\leftarrow y_i$ ($i=0,1,\hdots,N-1$) and $w_N\leftarrow z_N$, as well as $f(y_i)=\tilde{f}(y_i)+\tfrac{\mu}{2}\|y_i-x_\star\|^2$, it follows from Lemma~\ref{thm:potential} {(the weighted sums below are the same as that of Lemma~\ref{thm:potential}, written in terms of $\tilde{f}(\cdot)$ instead of $f(\cdot)$ and scaled by a factor $\frac{1}{L+\mu A_N}$; their reformulations are therefore also the same up to the rescaling)} that for $i=1,\hdots,N-1$
\begin{equation*}
\begin{aligned}
&{\lambda_{\star,i}^\star}\left[\tilde f(w_i)-\tilde f_\star+\langle \nabla \tilde f(w_i),w_\star-w_{i}\rangle +\tfrac1{2(L-\mu)}\|\nabla \tilde f(w_{i})\|^2\right]\\&\quad+{\lambda_{i-1,i}^\star}\left[ \tilde f(w_{i})- \tilde f(w_{i-1})+\langle \nabla \tilde f(w_{i});w_{i-1}-w_{i}\rangle +\tfrac1{2(L-\mu)}\|\nabla \tilde f (w_i)-\nabla \tilde f(w_{i-1})\|^2\right]=\frac1{L+\mu A_N}(\phi_{i+1}-\phi_{i}),
\end{aligned}
\end{equation*}
as well as
\begin{equation*}
\begin{aligned}
&{\lambda_{\star,0}^\star}\left[\tilde f(w_0)-\tilde f_\star+\langle\nabla \tilde f(w_0),w_\star-w_{0}\rangle +\tfrac1{2(L-\mu)}\|\nabla \tilde f (w_{0})\|^2\right]=\frac1{L+\mu A_N}(\phi_1-\phi_0),\\
&{\lambda_{N-1,\star}^\star}\left[\tilde f_\star-\tilde f(w_{N-1})+\tfrac1{2(L-\mu)}\|\nabla \tilde f(w_{N-1})\|^2\right]=-\frac{(1-q)A_N}{L+\mu A_N}\psi_{N-1}.
\end{aligned}
\end{equation*}
In addition, noting that $\phi_0=L\|w_0-w_\star\|^2$ allows reaching the following reformulation
\begin{equation*}
\begin{aligned}
&F\left(\sum_{i=0}^{N-2} {\lambda_{i,i+1}^\star}(\bff_{i+1}-\bff_i)+\sum_{i=0}^{N-1}{\lambda_{\star,i}^\star}\bff_i-{\lambda_{N-1,\star}^\star}\bff_{N-1}\right)+\mathrm{Tr}(KG)\\&=\frac{1}{L+\mu A_N}\left(\phi_0-(1-\cond)A_N\psi_{N-1}+\sum_{i=0}^{N-1}(\phi_{i+1}-\phi_i)\right)-\|z_N-w_\star\|^2\\&=0,
\end{aligned}
\end{equation*}
where the last equality follows from $\phi_N=(1-q)A_N\psi_{N-1}+(L+\mu A_N)\|z_N-w_\star\|^2$. Therefore, \shortMethodName{} is a solution to~\eqref{eq:minibound}. In more direct terms of the SDP~\eqref{eq:minibound}, one can verify that
\begin{equation*}
\begin{aligned}
& {\lambda_{\star,0}^\star}-{\lambda_{0,1}^\star}=\frac{1-\cond}{L}\left(\frac{A_{1}}{1+\cond A_N}-\frac{A_{1}}{1+\cond A_N}\right)=0&&\\
& {\lambda_{i-1,i}^\star}+{\lambda_{\star,i}^\star}-{\lambda_{i,i+1}^\star}=\frac{1-\cond }{L}\left(\frac{A_{i}}{1+\cond A_N}+\frac{A_{i+1}-A_i}{1+\cond A_N}-\frac{A_{i+1}}{1+\cond A_N}\right)=0 \quad &&\text{for }i=1,\hdots,N-2\\
& {\lambda_{N-2,N-1}^\star}+{\lambda_{\star,N-1}^\star}-{\lambda_{N-1,\star}^\star}=\frac{1-\cond }{L}\left(\frac{A_{N-1}}{1+\cond A_N}+\frac{A_{N}-A_{N-1}}{1+\cond A_N}-\frac{A_N}{1+\cond  A_N}\right)=0, &&
\end{aligned}
\end{equation*}
the previous computations therefore imply that $K=0$ for \shortMethodName, and hence $S''(\tau,\{\lambda_{i,j}\},{\{\beta_{i,j}\}})=\bxx_N\bxx_N^\top\succeq 0$.\qed

\section{An SDP formulation for optimizing function values}\label{a:func_values}
In this section, we show how to adapt the methodology developed in Section~\ref{s:optimized_steps} for a family of alternate design criteria, which include $(f(w_N)-f_\star)/\|w_0-w_\star\|^2$ and $(f(w_N)-f_\star)/(f(w_0)-f_\star)$ (for which numerical examples are provided respectively in Section~\ref{s:func1} and  Section~\ref{s:func2}). The developments {slightly differ from those required for optimizing ${\|w_N-w_\star\|^2}/{\|w_0-w_\star\|^2}$; and we decided not to present a unified version in the core of the text, for readability purposes.} {In particular, the set of selected inequalities is slightly different, altering the linearization procedure.}

The criteria we deal with in this section are of the form \[\frac{f(w_N)-f_\star}{c_w\|w_0-w_\star\|^2+c_f(f(w_0)-f_\star)}=\frac{\tilde f(w_N)- f_\star+\tfrac{\mu}{2}\|w_N-w_\star\|^2}{c_w\|w_0-w_\star\|^2+c_f(\tilde f(w_0)-f_\star+\tfrac{\mu}{2}\|w_0-w_\star\|^2)}.\]
As the steps are essentially the same as detailed in Section~\ref{s:optimized_steps}, we proceed without providing much detail. We start with the discrete version, using the set $I=\{\star,0,\hdots,N\}$
\begin{equation*}
\begin{aligned}
\max_{\substack{\{(w_i,g_i,f_i)\}_{i\in I}\\d\in\mathbb{N}}} & f_N-f_\star+\tfrac{\mu}{2}\|w_N-w_\star\|^2&&\\
\text{s.t. }& c_w\|w_0-w_\star\|^2+c_f(f_0-f_\star+\tfrac{\mu}{2}\|w_0-w_\star\|^2)=1,\, g_\star=0&&\\
&w_{k} \text{ generated by~\eqref{eq:method_reparam}} &&\text{for }k=1,\hdots,N\\
& f_i\geq f_j+\langle g_j;w_i-w_j\rangle +\tfrac1{2(L-\mu)}\|g_i-g_j\|^2\quad &&\text{for all }i,j\in I.
\end{aligned}
\end{equation*}
The upper bound we use is now very slightly different {(the selected subset of constraints is not the same as that of~\eqref{eq:relaxed})}
\begin{equation*}
\begin{aligned}
\mathrm{UB}_{\mu,L}(\{\alpha_{i,j}\})=\max_{\substack{\{(w_i,g_i,f_i)\}_{i\in I}\\d\in\mathbb{N}}} & f_N-f_\star+\tfrac{\mu}{2}\|w_N-w_\star\|^2&&\\
\text{s.t. }& c_w\|w_0-w_\star\|^2+c_f(f_0-f_\star+\tfrac{\mu}{2}\|w_0-w_\star\|^2)=1,\, g_\star=0&&\\
&w_{k} \text{ generated by~\eqref{eq:method_reparam} } &&\text{for }k=1,\hdots,N\\
& f_i\geq f_{i+1}+\langle g_{i+1};w_i-w_{i+1}\rangle +\tfrac1{2(L-\mu)}\|g_i-g_{i+1}\|^2 &&\text{for }i=0,\hdots,N-1\\
& f_\star\geq f_i+\langle g_i,w_\star-w_{i+1}\rangle +\tfrac1{2(L-\mu)}\|g_{i+1}\|^2 &&\text{for }i=0,\hdots,N
\end{aligned}
\end{equation*}
The corresponding SDP can be written using a similar couple $(G,F)$
\begin{equation*}
\begin{aligned}
G&=\begin{pmatrix} \|w_0-w_\star\|^2 & \langle g_0;w_0-w_\star\rangle & \langle g_1;w_0-w_\star\rangle &\hdots & \langle g_{N};w_0-w_\star\rangle\\
\langle g_0;w_0-w_\star\rangle & {\|g_0\|^2} & \langle g_1;g_0\rangle & \hdots &\langle g_{N};g_0\rangle\\
\langle g_1;w_0-w_\star\rangle & \langle g_1;g_0\rangle & {\|g_1\|^2} &\hdots &\langle g_{N};g_1\rangle\\
\vdots & \vdots & \vdots &\ddots & \vdots\\
\langle g_{N};w_0-w_\star\rangle & \langle g_{N};g_0\rangle & \langle g_{N};g_1\rangle &\hdots &\|g_{N}\|^2\\
\end{pmatrix}\\
F&=\begin{pmatrix}f_0-f_\star \\ f_1-f_\star \\ \vdots \\ f_{N}-f_\star \end{pmatrix},
\end{aligned}
\end{equation*}
and the similar notations
\[ \bxx_0=e_1\in\mathbb{R}^{N+2},\quad \bgg_i=e_{i+2}\in\mathbb{R}^{N+2},\quad \bff_i=e_{i+1}\in\mathbb{R}^{N+1},\]
with $i=0,\hdots,N$ and $e_i$ being the unit vector whose $i$th component is equal to $1$. In addition, we can also denote by
\[\bxx_k=\bxx_0\left(1-\tfrac{\mu}{L}\sum_{i=0}^{k-1}\alpha_{k,i}\right)-\sum_{i=0}^{k-1}\tfrac{\alpha_{k,i}}{L}\bgg_i.\]
A dual formulation of $\mathrm{UB}_{\mu,L}$ is given by (we directly included the Schur complement)
\begin{equation*}
\begin{aligned}
\mathrm{UB}_{\mu,L}(\{\alpha_{i,j}\})=\min_{\tau,\lambda_{i,j}\geq0}& \tau,&&\\
\text{s.t.}&\,\begin{pmatrix}\varS'(\tau,\{\lambda_{i,j}\},\{{\alpha_{i,j}}\}) & \sqrt{\mu}\bxx_N \\ \sqrt{\mu}\bxx_N^\top & 2\end{pmatrix}\succeq 0,&&\\
&\tau \, c_f\, \bff_{0}+\sum_{i=0}^{N-1} \lambda_{i,i+1}(\bff_{i+1}-\bff_i)+\sum_{i=0}^{N}\lambda_{\star,i}\bff_i=\bff_{N}
\end{aligned}
\end{equation*}
with \begin{equation*}
\begin{aligned}
\varS'(\tau,\{\lambda_{i,j}\},\{{\alpha_{i,j}}\})=&\,\tau\, (c_w+c_f\tfrac{\mu}{2}) \bxx_0\bxx_0^\top+\sum_{i=0}^{N} \frac{\lambda_{\star,i}}{2}\left(-\bgg_i \bxx_i^\top -\bxx_i\bgg_i^\top+\tfrac{1}{L-\mu}\bgg_i\bgg_i^\top\right)\\
&+\sum_{i=0}^{N-1}\frac{\lambda_{i,i+1}}{2}\left(\bgg_{i+1}(\bxx_i-\bxx_{i+1})^\top+(\bxx_i-\bxx_{i+1})\bgg_{i+1}^\top+\tfrac{1}{L-\mu}(\bgg_i-\bgg_{i+1})(\bgg_i-\bgg_{i+1})^\top\right).
\end{aligned}
\end{equation*} Note that that the equality constraint corresponds to
\begin{equation*}
\begin{aligned}
& c_f+\lambda_{\star,0}-\lambda_{0,1}=0&&\\
& \lambda_{i-1,i}+\lambda_{\star,i}-\lambda_{i,i+1}=0 \quad &&\text{for }i=1,\hdots,N-1\\
& \lambda_{N-1,N}+\lambda_{\star,N}=1. &&
\end{aligned}
\end{equation*}
We perform a some additional work on $\varS'$ {(whose dependency on $\{\alpha_{i,j}\}$ is implicit through the dependency on $\{\bxx_i\}$)}, as before
\begin{equation*}
\begin{aligned}
\varS'(\tau,\{\lambda_{i,j}\},\{{\alpha_{i,j}}\})=&\,\tau\, (c_w+c_f\tfrac{\mu}{2}) \bxx_0\bxx_0^\top+\frac{1}{2(L-\mu)}\left(\sum_{i=0}^{N} {\lambda_{\star,i}}\bgg_i\bgg_i^\top+\sum_{i=0}^{N-1}{\lambda_{i,i+1}}(\bgg_i-\bgg_{i+1})(\bgg_i-\bgg_{i+1})^\top\right)\\&-\sum_{i=0}^{N} {\lambda_{\star,i}}\tfrac12\left(\bgg_i \bxx_i^\top +\bxx_i\bgg_i^\top\right)
+\sum_{i=0}^{N-1}{\lambda_{i,i+1}}\tfrac12\left(\bgg_{i+1}(\bxx_i-\bxx_{i+1})^\top+(\bxx_i-\bxx_{i+1})\bgg_{i+1}^\top\right)\\
=&\,\tau\, (c_w+c_f\tfrac{\mu}{2}) \bxx_0\bxx_0^\top+\frac{1}{2(L-\mu)}\left(\sum_{i=0}^{N} {\lambda_{\star,i}}\bgg_i\bgg_i^\top+\sum_{i=0}^{N-1}{\lambda_{i,i+1}}(\bgg_i-\bgg_{i+1})(\bgg_i-\bgg_{i+1})^\top\right)\\&-\lambda_{\star,0}\tfrac12(\bgg_0\bxx_0^\top+\bxx_0\bgg_0^\top)-\sum_{i=1}^N(\lambda_{\star,i}+\lambda_{i-1,i})\tfrac12(\bgg_i\bxx_i^\top+\bxx_i\bgg_i^\top)\\
&+\sum_{i=0}^{N-1}\lambda_{i,i+1}\tfrac12(\bgg_{i+1}\bxx_i^\top+\bxx_i\bgg_{i+1}^\top)\\
=&\,\tau\, (c_w+c_f\tfrac{\mu}{2}) \bxx_0\bxx_0^\top+\frac{1}{2(L-\mu)}\left(\sum_{i=0}^{N} {\lambda_{\star,i}}\bgg_i\bgg_i^\top+\sum_{i=0}^{N-1}{\lambda_{i,i+1}}(\bgg_i-\bgg_{i+1})(\bgg_i-\bgg_{i+1})^\top\right)\\&-\lambda_{\star,0}\tfrac12(\bgg_0\bxx_0^\top+\bxx_0\bgg_0^\top)-\sum_{i=1}^{N-1}\lambda_{i,i+1}\tfrac12(\bgg_i\bxx_i^\top+\bxx_i\bgg_i^\top)-\tfrac12(\bgg_N\bxx_N^\top+\bxx_N\bgg_N^\top)\\&+\sum_{i=0}^{N-1}\lambda_{i,i+1}\tfrac12(\bgg_{i+1}\bxx_i^\top+\bxx_i\bgg_{i+1}^\top),
\end{aligned}
\end{equation*}
where we used $\lambda_{\star,i}+\lambda_{i-1,i}=\lambda_{i,i+1}$ (for $i=1,\hdots,N-1$) and $\lambda_{\star,N}+\lambda_{N-1,N}=1$. Now, making the dependence on $\alpha_{i,j}$'s explicit again, we arrive to
\begin{equation*}
\begin{aligned}
\varS'(\tau,\{\lambda_{i,j}\},\{{\alpha_{i,j}}\})=&\,\tau\, (c_w+c_f\tfrac{\mu}{2}) \bxx_0\bxx_0^\top+\frac{1}{2(L-\mu)}\left(\sum_{i=0}^{N} {\lambda_{\star,i}}\bgg_i\bgg_i^\top+\sum_{i=0}^{N-1}{\lambda_{i,i+1}}(\bgg_i-\bgg_{i+1})(\bgg_i-\bgg_{i+1})^\top\right)\\&-\lambda_{\star,0}\tfrac12(\bgg_0\bxx_0^\top+\bxx_0\bgg_0^\top)-\sum_{i=1}^{N-1}\lambda_{i,i+1}\left(1-\tfrac{\mu}{L}\sum_{j=0}^{i-1}\alpha_{i,j}\right)\tfrac12(\bgg_i\bxx_0^\top+\bxx_0\bgg_i^\top)\\&+\sum_{i=1}^{N-1}\lambda_{i,i+1}\sum_{j=0}^{i-1}\alpha_{i,j}\tfrac12(\bgg_i\bgg_j^\top+\bgg_j\bgg_i^\top)-\left(1-\tfrac{\mu}{L}\sum_{j=0}^{N-1}\alpha_{N,j}\right)\tfrac12(\bgg_N\bxx_0^\top+\bxx_0\bgg_N^\top)\\&+\sum_{j=0}^{N-1}\alpha_{N,j}\tfrac12(\bgg_N\bgg_j^\top+\bgg_j\bgg_N^\top)+\sum_{i=0}^{N-1}\lambda_{i,i+1}\left(1-\tfrac{\mu}{L}\sum_{j=0}^{i-1}\alpha_{i,j}\right)\tfrac12(\bgg_{i+1}\bxx_0^\top+\bxx_0\bgg_{i+1}^\top)\\&-\sum_{i=0}^{N-1}\lambda_{i,i+1}\sum_{j=0}^{i-1}\alpha_{i,j}\tfrac12(\bgg_{i+1}\bgg_j^\top+\bgg_j\bgg_{i+1}^\top)
\end{aligned}
\end{equation*}
and it remains to remark that the change of variables
\begin{equation}\label{eq:design_SDP_func}
\begin{aligned}
{\beta_{i,j}}=\left\{\begin{array}{ll}
\lambda_{i,i+1}\alpha_{i,j}   \quad  & \text{if } 0\leq i \leq N-1 \\
 \alpha_{N,j} &\text{if } i=N.
\end{array}\right.
\end{aligned}    
\end{equation}
linearizes the bilinear matrix inequality, again, and it remains to solve the SDP~\eqref{eq:min_upperbound} using standard packages. Numerical results for the pairs $(c_w,c_f)=(1,0)$ and $(c_w,c_f)=(0,1)$ are respectively provided in Section~\ref{s:func1} and Section~\ref{s:func2}. A source code for implementing those SDP is provided in Section~\ref{s:ccl}.

\section{Numerical examples} \label{a:numerics}
As shown in Appendix~\ref{a:func_values}, slight modifications of the relaxations used for obtaining~\eqref{eq:minibound} allows forming tractable problems for optimizing the parameters of fixed-step methods under different optimality criteria.
Although we were unable to obtain closed-form solutions to the problems arising for these alternative criteria, the resulting problems can still be approximated numerically for specific values of $\mu$, $L$ and $N$.

In the following, we provide a couple of examples that were obtained by numerically solving the first-order method design problem~\eqref{eq:minibound}, formulated as a linear semidefinite program using standard solvers~\citep{Yalmip,Article:Mosek}.

\subsection{Optimized methods for $(f(w_N)-f_\star)/{\|w_0-w_\star\|^2}$}\label{s:func1}

As a first example, we consider the criterion $({f(w_N)-f_\star})/{\|w_0-w_\star\|}$. 
The following list provides solutions obtained by solving the corresponding design problem for $N=1,...,5$ with $L=1$ and $\mu=.1$. The solutions are presented using the notations from~\eqref{eq:method} together with the corresponding worst-case guarantees. 
\begin{itemize}
    \item {For a single iteration, {by solving the corresponding optimization problem,} we obtain a method with guarantee $\tfrac{f(w_1)-f_\star}{\|w_0-w_\star\|} \leq 0.1061$ and step size
    \[ [{h_{i,j}^\star}]=\begin{bmatrix} 1.4606  \end{bmatrix}.\]
    This bound and the corresponding step size match the optimal step size $h_{1,0}=\tfrac{\cond+1-\sqrt{\cond^2-\cond+1}}{\cond}$, see~\citep[Theorem 4.14]{taylorconvex}.}
    \item For $N=2$ iterations, we obtain $\tfrac{f(w_2)-f_\star}{\|w_0-w_\star\|} \leq 0.0418$ with
    \[ [{h_{i,j}^\star}]=\begin{bmatrix}
 1.5567 &  \\
 0.1016 & \,\,\,1.7016   \end{bmatrix}.\]
    \item For $N=3$, we obtain $\tfrac{f(w_3)-f_\star}{\|w_0-w_\star\|} \leq 0.0189$ with
    \[ [{h_{i,j}^\star}]=\begin{bmatrix}  
 1.5512 &  &   \\
 0.1220 & \,\,\,1.8708 &   \\
 0.0316 & \,\,\,0.2257 & \,\,\,1.8019  \end{bmatrix}.\]
    \item For $N=4$, we obtain $\tfrac{f(w_4)-f_\star}{\|w_0-w_\star\|} \leq 0.0089$, with
    \[ [{h_{i,j}^\star}]=\begin{bmatrix} 
 1.5487 &  &  &   \\
 0.1178 & \,\,\,1.8535 &  &  \\
 0.0371 & \,\,\,0.2685 & \,\,\,2.0018 &   \\
 0.0110 & \,\,\,0.0794 & \,\,\,0.2963 & \,\,\,1.8497  \end{bmatrix}.\]
    \item Finally, for $N=5$, we obtain $\tfrac{f(w_5)-f_\star}{\|w_0-w_\star\|}\leq 0.0042$ with
    \[ [{h_{i,j}^\star}]=\begin{bmatrix}
 1.5476 &  &  &  &  \\
 0.1159 & \,\,\,1.8454 &  &  &   \\
 0.0350 & \,\,\,0.2551 & \,\,\,1.9748 &  &  \\
 0.0125 & \,\,\,0.0913 & \,\,\,0.3489 & \,\,\,2.0625 &   \\
 0.0039 & \,\,\,0.0287 & \,\,\,0.1095 & \,\,\,0.3334 & \,\,\,1.8732   \end{bmatrix}.\]
\end{itemize}
{Note that when $\mu=0$, we recover the step size policy of the OGM by~\citet{kim2015optimized}. When setting $\mu>0$, we observe that the resulting optimized method is apparently less practical as the step sizes critically depend on the horizon $N$. In particular, one can observe that {$h_{1,0}^\star$} varies with the horizon $N$.} 

Figure~\ref{fig:LB_vs_Opt} illustrates the behavior of the worst-case guarantee for larger values of $N$ and compares it to the currently best known corresponding lower bound, as well as to worst-case guarantees for TMM, Nesterov's Fast Gradient Method (FGM) for strongly convex functions, as well as to the methods generated with the SSEP procedure from~\citep{drori2020efficient}. All the worst-case guarantees are computed numerically using the corresponding performance estimation problems (see e.g., the toolbox~\citep{taylor2017performance}), {and as a result, they are tight in the sense that matching inputs to the algorithms attaining the bounds can be numerically constructed}.
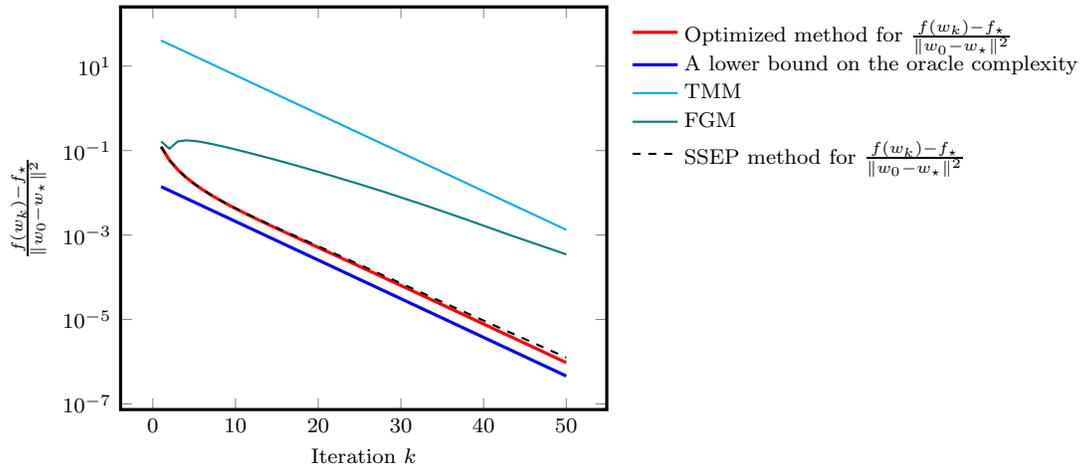
\begin{figure}
\centering
		\begin{tikzpicture}
			\begin{semilogyaxis}[legend pos=outer north east,legend style={draw=none},legend cell align={left},plotOptions,width=.5\linewidth]
			\addplot [red] table [x=k,y=WC]  {Comparison_L1_m01.dat};
			\addplot [blue] table [x=k,y=LB]  {Comparison_L1_m01.dat};
			\addlegendentry{Optimized method for $\tfrac{f(w_k)-f_\star}{\|w_0-w_\star\|^2}$}
			\addlegendentry{A lower bound on the oracle complexity}
			\addplot[cyan,thick] table [y=TMM, x=N, skip coords between index={21}{30}]{Comparisons_L1_m01_GFOM.txt};
			\addlegendentry{TMM}
			\addplot[teal,thick] table [y=FGM, x=N, skip coords between index={21}{30}]{Comparisons_L1_m01_GFOM.txt};
			\addlegendentry{FGM}
			\addplot[dashed,black,thick] table [y=SSEP, x=N, skip coords between index={21}{30}]{Comparisons_L1_m01_GFOM.txt};
			\addlegendentry{SSEP method for $\tfrac{f(w_k)-f_\star}{\|w_0-w_\star\|^2}$}
			\end{semilogyaxis}
		\end{tikzpicture}
		\caption{Numerical comparison (for $L=1$, $\mu=0.01$) between (i) the worst-case guarantee of the optimized method for $\tfrac{f(w_k)-f_\star}{\|w_0-w_\star\|^2}$ (in red; obtained from developments in Appendix~\ref{a:func_values}, and numerical examples in Appendix~\ref{s:func1}); (ii)~a lower bound on the oracle complexity for this setup (in blue; presented in~\citep[Corollary 3]{drori2021exact}), which corresponds to $\tfrac{f(w_k)-f_\star}{\|w_0-w_\star\|^2}\geq \mu\tfrac{2-\sqrt{\cond}}{1+\sqrt{\cond}}\left(1-\sqrt{\cond}\right)^{2k}$; (iii) the triple momentum method~\citep{van2017fastest} (cyan); (iv) Nesterov's fast gradient method (defined in~\citep[Section 2.2, ``Constant Step Scheme, II'']{nest-book-04}; FGM, green), and (v) the method generated by the subspace-search elimination procedure (SSEP) from~\citep{drori2020efficient} (dashed, black).} \label{fig:LB_vs_Opt}
\end{figure}
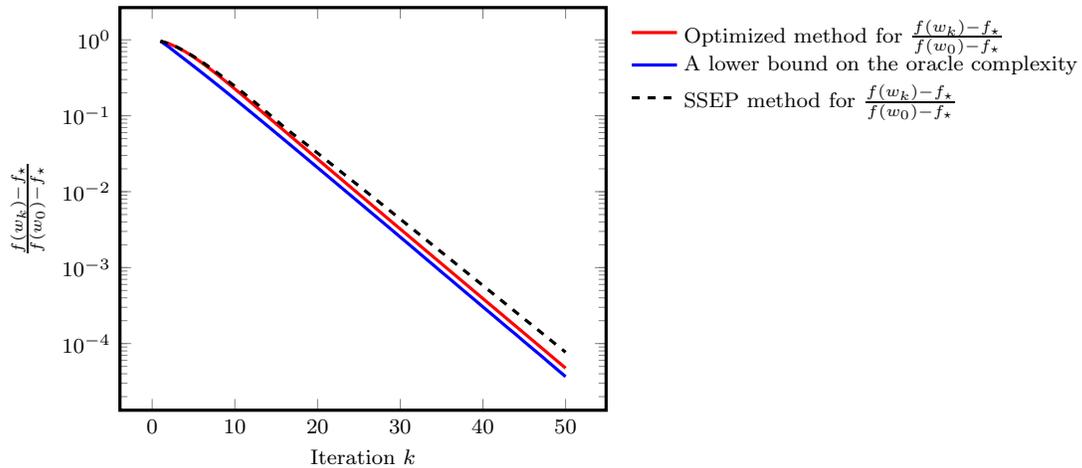
\begin{figure}
\centering
		\begin{tikzpicture}
			\begin{semilogyaxis}[legend pos=outer north east,legend style={draw=none},legend cell align={left},plotOptions2,width=.5\linewidth]
			\addplot [red] table [x=k,y=WC]  {Comparison_L1_m01_funcfunc.dat};
			\addplot [blue] table [x=k,y=LBnew]  {Comparison_L1_m01_funcfunc.dat};
			\addlegendentry{Optimized method for $\tfrac{f(w_k)-f_\star}{f(w_0)-f_\star}$}
			\addlegendentry{A lower bound on the oracle complexity}
			\addplot [black, dashed] table [x=k,y=WC_SSEP]  {Comparison_L1_m01_GFOM_funcfunc.dat};
			\addlegendentry{SSEP method for $\tfrac{f(w_k)-f_\star}{f(w_0)-f_\star}$}
			
			\end{semilogyaxis}
		\end{tikzpicture}
		\caption{Numerical comparison (for $L=1$, $\mu=0.01$) between (i) the worst-case guarantee of the optimized method for $\tfrac{f(w_k)-f_\star}{f(w_0)-f_\star}$ (in red; obtained from developments in Appendix~\ref{a:func_values}, and numerical examples in Appendix~\ref{s:func2}); (ii)~a lower bound on the oracle complexity for this setup (in blue; computed numerically using the procedure from~\citep{drori2021exact}); and (iii) a method generated by the subspace-search elimination procedure (SSEP) from~\citep{drori2020efficient} (dashed, black).} 
		\label{fig:LB_vs_Opt_f}
\end{figure}
\clearpage
\subsection{Optimized methods for $({f(w_N)-f_\star})/({f(w_0)-f_\star})$}\label{s:func2}
As in the previous section, the technique can be adapted for the criterion $(f(w_N)-f_\star)/(f(x_0)-f_\star)$, see  Appendix~\ref{a:func_values} for details. The following step sizes were obtained by setting $L=1$ and $\mu=.1$ and solving the resulting optimization problem from different values of $N$.
\begin{itemize}
    \item For a single iteration, $N=1$, we obtain a guarantee $\tfrac{f(w_1)-f_\star}{f(w_0)-f_\star} \leq 0.6694$ with the corresponding step size
    \[ [{h_{i,j}^\star}]=\begin{bmatrix} 1.8182   \end{bmatrix},\]
    which matches the known optimal step size $2/(L+\mu)$ for this setup~\citep[Theorem 4.2]{de2017worst}.
    \item For $N=2$, we obtain $\tfrac{f(w_2)-f_\star}{f(w_0)-f_\star} \leq 0.3554$ with
    \[ [{h_{i,j}^\star}]=\begin{bmatrix} 2.0095 &   \\ 0.4229 & \,\,\,2.0095 \end{bmatrix}.\]
    \item For $N=3$, we obtain $\tfrac{f(w_3)-f_\star}{f(w_0)-f_\star} \leq  0.1698$ with
    \[ [{h_{i,j}^\star}]=\begin{bmatrix} 1.9470 &  &  \\
 0.4599 & \,\,\,2.2406 &   \\
 0.1705 & \,\,\,0.4599 & \,\,\,1.9470 \end{bmatrix}.\]
    \item For $N=4$, we obtain $\tfrac{f(w_4)-f_\star}{f(w_0)-f_\star} \leq 0.0789$ with
    \[ [{h_{i,j}^\star}]=\begin{bmatrix}
 1.9187 &  &  &   \\
 0.4098 & \,\,\,2.1746 &  &  \\
 0.1796 & \,\,\,0.5147 & \,\,\,2.1746 &  \\
 0.0627 & \,\,\,0.1796 & \,\,\,0.4098 & \,\,\,1.9187   \end{bmatrix}.\]
    \item Finally, for $N=5$, we reach $\tfrac{f(w_5)-f_\star}{f(w_0)-f_\star} \leq 0.0365$ with
    \[ [{h_{i,j}^\star}]=\begin{bmatrix} 
 1.9060 &  &  &  &   \\
 0.3879 & \,\,\,2.1439 &  &  &   \\
 0.1585 & \,\,\,0.4673 & \,\,\,2.1227 &  &   \\
 0.0660 & \,\,\,0.1945 & \,\,\,0.4673 & \,\,\,2.1439 &   \\
 0.0224 & \,\,\,0.0660 & \,\,\,0.1585 & \,\,\,0.3879 & \,\,\,1.9060  \end{bmatrix}.\]
\end{itemize}
Note that the resulting method is again apparently less practical than \shortMethodName, as step sizes also critically depend on the horizon $N$; for example, observe again that the value of $h_{1,0}^\star$ depends on $N$. Interestingly, one can observe that the corresponding step sizes are symmetric, and that the worst-case guarantees seem to behave slightly better than in the distance problem ${\|w_N-w_\star\|^2}/{\|w_0-w_\star\|^2}$, although their asymptotic rate has to be the same, due to 
the properties of strongly convex functions.   Figure~\ref{fig:LB_vs_Opt_f} illustrates the worst-case guarantees of the corresponding method for larger numbers of iterations, and compares it to the lower bound.

\end{document}